\documentclass{amsart}
\usepackage{amsmath,amssymb,amsthm,mathrsfs,mathtools}
\usepackage[colorlinks]{hyperref}

\usepackage[backend=bibtex,doi=false,isbn=false,url=false,backref=true]{biblatex}
\newtheorem{dfn}{Definition}[section]
\newtheorem{thm}[dfn]{Theorem}
\newtheorem{prop}[dfn]{Proposition}
\newtheorem{lem}[dfn]{Lemma}
\newtheorem{cor}[dfn]{Corollary}

\numberwithin{equation}{section}

\title[Stable determination of the potential in the high frequency limit]{Stable determination of the potential for the Helmholtz equation in the high frequency limit from boundary measurements}

\author{Mourad Choulli}
\address{Universit\'{e} de Lorraine, 34 cours L\'{e}opold, 54052 Nancy cedex, France}
\email{mourad.choulli@univ-lorraine.fr}

\author{Hiroshi Takase}
\address{Institute of Mathematics for Industry, Kyushu University, 744 Motooka, Nishi-ku, Fukuoka 819-0395, Japan}
\email{htakase@imi.kyushu-u.ac.jp}

\date{\today}

\begin{document}
\begin{abstract}
We establish a triple logarithmic stability estimate of determining the potential in a Helmholtz equation from a partial Dirichlet-to-Neumann map in the high frequency limit. This estimate is proved under the assumption that the potential is known near the boundary of a domain when the dimension is greater than or equal to $3$. In addition, we show a triple logarithmic stability for an interior impedance problem.
\end{abstract}
\maketitle

\section{Introduction and main result}\label{section1}

Let $n\ge 3$ be an integer, $\Omega\subset \mathbb{R}^n$ be a $C^{1,1}$ bounded domain and $(g_{jk})\in W^{1,\infty}(\Omega;\mathbb{R}^{n\times n})$ be a symmetric matrix-valued function satisfying, for some $\delta_0>0$
\[
\sum_{j,k=1}^n g_{jk}\eta^j\eta^k\ge\delta_0|\eta|^2,\quad x\in\Omega,\; \eta=(\eta^1,\ldots,\eta^n)\in\mathbb{R}^n.
\]
Note that $(g^{jk})$ the matrix inverse to $(g_{jk})$ is uniformly positive definite as well. Throughout this paper, $\Delta_g$ denotes the Laplace-Beltrami operator associated with the metric tensor $g=\sum_{j,k=1}^ng_{jk}dx^j\otimes dx^k$ given by
\[
\Delta_g u:=\left(1/{\sqrt{|g|}}\right)\sum_{j,k=1}^n\partial_j\left(\sqrt{|g|}g^{jk}\partial_k u\right),
\]
where $|g|:=\mbox{det}(g)$. When $g=\mathbf{I}$, the flat metric, we write simply $\Delta$ for the standard Laplacian $\Delta_\mathbf{I}$. 

The Riemannian measures on $\Omega$ and $\partial \Omega$ will denoted respectively by $d\mu:=\sqrt{|g|}dx$ and $ds:=\sqrt{|g|}d\sigma$, where $d\sigma$ is the surface measure on $\partial \Omega$.

For $q\in L^\infty(\Omega;\mathbb{R})$, consider the operator
\[
A_q u:=(-\Delta_g+q)u,\quad u\in H^2(\Omega)\cap H^1_0(\Omega).
\]
Let $\sigma(A_q)$ denotes the  spectrum of $A_q$. It is well known that $\sigma(A_q)$ consists of eigenvalues. Recall that $R_q(\lambda):=(A_q-\lambda)^{-1}$, the resolvent of $A_q$, belongs to $\mathscr{B}(L^2(\Omega))$ for all $\lambda\in \rho(A_q):=\mathbb{C}\setminus\sigma(A_q)$, the resolvent set of $A_q$.

Fix $q_0\in L^\infty(\Omega;\mathbb{R})$ and $\kappa_0>0$ arbitrarily and set $\kappa:=\|q_0\|_{L^\infty(\Omega)}+\kappa_0$. Define for $\lambda\in \rho(A_{q_0})$
\[
\mathscr{Q}_\lambda:=\left\{q=q_0+q'\in L^\infty(\Omega;\mathbb{R});\; \|q'\|_{L^\infty(\Omega)}<\min\left(\|R_{q_0}(\lambda)\|^{-1},\kappa_0\right)\right\},
\]
where $\|R_{q_0}(\lambda)\|$ is the operator norm of $R_{q_0}(\lambda)$ in $\mathscr{B}(L^2(\Omega))$. For all $\lambda\in\rho(A_{q_0})$ and $q\in\mathscr{Q}_\lambda$, we verify that $\|q\|_{L^\infty(\Omega)}\le \kappa$ and 
\[
\lambda\in\rho(A_{q_0})\Longrightarrow\lambda\in \bigcap_{q\in\mathscr{Q}_\lambda}\rho(A_q)
\]
(see the proof of Lemma \ref{Lem2.1_CAC_2023}). In particular, $R_q(\lambda)\in \mathscr{B}(L^2(\Omega))$ for all $\lambda\in\rho(A_{q_0})$ and $q\in\mathscr{Q}_\lambda$. Set for $\lambda\in\rho(A_{q_0})$
\[
\sigma_\lambda:=\bigcup_{q\in\mathscr{Q}_\lambda}\sigma(A_q)
\]
and
\[
\mathbf{e}_\lambda:=\max\left(\frac{1}{\mbox{dist}(\lambda,\sigma_\lambda)},1\right)(\ge 1).
\]

In the rest of this text, $\Gamma$ and $\Sigma$ will be two closed subsets of $\partial\Omega$ with nonempty interior, and $\Omega_0\Subset\Omega$ will denote a $C^{1,1}$ open set such that $\Omega_1:=\Omega\setminus \overline{\Omega_0}$ is connected.

Define the closed subspace of $H^{3/2}(\partial\Omega)$ given by
\[
H^{3/2}_\Gamma(\partial\Omega):=\left\{\varphi\in H^{3/2}(\partial\Omega);\; \mbox{supp}(\varphi)\subset \Gamma\right\}
\]
This space is endowed with the norm of $H^{3/2}(\partial\Omega)$. Also, define
\[
H^{1/2}(\Sigma):=\{\varphi=h_{|\Sigma};\; h\in H^{1/2}(\partial\Omega)\}
\]
endowed with the quotient norm
\[
\|\varphi\|_{H^{1/2}(\Sigma)}:=\inf\{\|h\|_{H^{1/2}(\partial\Omega)};\; h_{|\Sigma}=\varphi\}.
\]

For all  $\lambda\in\rho(A_{q_0})$, $q\in\mathscr{Q}_\lambda$ and $\varphi\in H^{3/2}_\Gamma(\partial\Omega)$, the BVP
\begin{equation}\label{BVP}
\begin{cases}
(-\Delta_g+q-\lambda)u=0\quad \text{in}\; \Omega,
\\
u_{|\partial\Omega}=\varphi.
\end{cases}
\end{equation}
admits a unique solution $u=u_{q,\lambda}(\varphi)\in H^2(\Omega)$. Define the partial Dirichlet-to-Neumann maps $\Lambda_{q,\lambda}$ given by
\[
\Lambda_{q,\lambda}\colon H^{3/2}_\Gamma(\partial\Omega)\ni \varphi\mapsto \partial_{\nu_g} u_{q,\lambda}(\varphi)_{|\Sigma}\in H^{1/2}(\Sigma),
\]
where
\[
\partial_{\nu_g}=|\nu|_g^{-1}\sum_{j,k=1}^n \nu_jg^{jk}\partial_k,\quad |\nu|_g:=\sqrt{\sum_{j,k=1}^n g^{jk}\nu_j\nu_k},
\]
$\nu=(\nu_1,\ldots ,\nu_n)$ being the unit normal vector field to $\partial \Omega$.

From the usual elliptic $H^2(\Omega)$ a priori estimates, there exist two constants $c_0=c_0(n,\Omega,g)>0$ and $C_0=C_0(n,\Omega,g,\kappa,\lambda)>0$ such that
\begin{align*}
&\|\partial_{\nu_g} u_{q,\lambda}(\varphi)_{|\Sigma}\|_{H^{1/2}(\Sigma)} \le \|\partial_{\nu_g} u_{q,\lambda}(\varphi)\|_{H^{1/2}(\partial \Omega)}
\\
&\hskip 5cm \le c_0\|u_{q,\lambda}(\varphi)\|_{H^2(\Omega)}\le C_0\|\varphi\|_{H^{3/2}(\partial \Omega)}.
\end{align*}
In other words, $\Lambda_{q,\lambda}\in\mathscr{B}(H^{3/2}_\Gamma(\partial\Omega); H^{1/2}(\Sigma))$.

We are mainly interested in the problem of determining $q$ from $\Lambda_{q,\lambda}$. We establish a stability inequality that explicitly states the dependence of the positive spectral parameter $\lambda$ on the constants.

Set
\[
\mathbf{b}_\lambda:=\sqrt{2\cosh(\sqrt{\lambda}/2)}\left(\ge \sqrt{2}\right),\quad \lambda >0.
\]

For all $c>0$, set $\mathfrak{e}(c)=e^{e^{e^c}}$, $L(r)=\log \log \log r$, $r>\mathfrak{e}(c)$ and define $\Phi_c\colon(0,\infty)\rightarrow\mathbb{R}$ by
\[
\Phi_{c} (r):=r^{-1}\chi_{\left]0,\mathfrak{e}(c)\right]}(r)+L(r)^{-2/(n+2)}\chi_{\left]\mathfrak{e}(c),\infty\right[}(r),
\]
where $\chi_J$ denotes the characteristic function of an interval $J\subset\mathbb{R}$.

%Our first main result is the following theorem.

The results in Sections \ref{section2}, \ref{section3} and \ref{section4} hold for nonflat metric $g$. However, in the proof of Section \ref{section5}, we assume flatness $g=\mathbf{I}$ in order to use the complex geometric optics solutions constructed in Appendix \ref{sectionC}. Therefore, in the first main result described next, we assume $g=\mathbf{I}$. Note that although the Dirichlet-to-Neumann map $\Lambda_{q,\lambda}$ and the function spaces $\mathscr{S}^j_{q,\lambda}$, $j=0,1$ appearing in Section \ref{section3} depend on $g$, we have avoided explicitly stating this to avoid cumbersome notation.

\begin{thm}\label{triple_log_stability}
Let $n\ge 3$ be an integer. Assume $g=\mathbf{I}$ and set $\varsigma=(n,\Omega,\Omega_0,\kappa)$. There exist constants $C=C(\varsigma,\Gamma,\Sigma)>0$, $c=c(\varsigma,\Gamma)>0$ and $\theta=\theta(\varsigma,\Sigma)\in(0,1)$ such that for all $\lambda\in\rho(A_{q_0})$ so that $\lambda\ge 1$ and for all $q_j\in\mathscr{Q}_\lambda$, $j=1,2$, satisfying $q_1=q_2$ in $\Omega_1$ we have
\[
C\|q_1-q_2\|_{H^{-1}(\Omega)}
\le \lambda^5\mathbf{e}_\lambda^3\mathbf{b}_\lambda\Phi_{c}\left(\|\Lambda_{q_1,\lambda}-\Lambda_{q_2,\lambda}\|^{-\theta}\right).
\]
\end{thm}

We emphasize that the condition $\lambda \ge 1$ can be replaced by $\lambda\ge \lambda_0$, where $\lambda_0$ is arbitrarily fixed. In this case, the constant $C$ in the statement of Theorem \ref{triple_log_stability} also depends on $\lambda_0$.

Theorem \ref{triple_log_stability} gives a precise dependence of the stability constant on the sufficiently large spectral parameter $\lambda$. It is worth noting that the stability deteriorates not only for high frequency but also when the square of the frequency is close to $\sigma_\lambda$. The difference between the stability in Theorem \ref{triple_log_stability} and those obtained in \cite{GarciaFerrero2022, Krupchyk2019, Ruland2019} is that the dependence of $\lambda$ appearing in the constant of modulus on continuity is explicitly given. Double logarithmic stability inequality has already been obtained  by the first author \cite[Theorem 1.1]{Choulli2023a} in the case $n=3$.

We also point out that we have a single logarithmic stability inequality  for the partial Dirichlet-to-Neumann map given by
\[
H^{3/2}(\partial\Omega)\ni \varphi\mapsto \partial_{\nu_g} u_{q,\lambda}(\varphi)_{|\Sigma}\in H^{1/2}(\Sigma),
\]
This result was established by the first author in \cite[Theorem 1.6]{Choulli2023a}.

The rest of this text is organized as follows. In Section \ref{section2} we give some results for  global quantitative uniqueness of continuation adapted to our problem. Using these results, we obtain a quantitative Runge approximation result in Section \ref{section3}. Section \ref{section4} is devoted to establishing a preliminary inequality by means of the result in Section \ref{section3}. In Section \ref{section5}, we prove Theorem \ref{triple_log_stability} by using the integral inequality in Section \ref{section4} and complex geometric optic solutions. Finally, in Section \ref{section6}, we modify the analysis we performed in the case of the Dirichlet boundary condition to extend Theorem \ref{triple_log_stability} to the case of an impedance boundary condition.
 Auxiliary results are presented in appendices.

\section{Global quantitative uniqueness of continuation}\label{section2}

Define
\[
H^1_{0,\Sigma}(\Omega):=\{u\in H^1(\Omega);\; u_{|\Sigma}=0\}
\]
and, for notational simplicity, set $\zeta:=(n,\Omega,g,\kappa)$

\begin{thm}\label{Thm2.3_CAC_2023}
There exist constants $C=C(\zeta,\Sigma)>0$ and $c=c(\zeta,\Sigma)>0$ such that for all $\lambda\ge 1$, $q\in L^\infty(\Omega)$ satisfying $\|q\|_{L^\infty(\Omega)}\le\kappa$, $0<\varepsilon<1$ and $u\in H^2(\Omega)\cap H^1_{0,\Sigma}(\Omega)$ we have
\[
C\|u\|_{H^1(\Omega)}\le\lambda\varepsilon^{1/8} \|u\|_{H^2(\Omega)}+e^{e^{c/\varepsilon}}\left(\|(\Delta_g-q+\lambda)u\|_{L^2(\Omega)}+\|\partial_{\nu_g} u\|_{L^2(\Sigma)}\right).
\]
\end{thm}
\begin{proof}
We will apply a result of quantitative uniqueness of continuation and must represent the dependence of the spectral parameter $\lambda$ on constants explicitly. To do this, we apply the result by Choulli \cite{Choulli2020} to an elliptic operator $-\Delta_g-\partial_t^2+q$ without the spectral parameter in the Lipschitz domain $\Omega\times(0,1)$.

According to \cite[Theorem 1]{Choulli2020} with $j=1$ and $s=1/4$, there exist constants $C=C(\zeta,\Sigma)>0$ and $c=c(\zeta,\Sigma)>0$, for any  $0<\varepsilon<1$ and $v\in H^2(\Omega\times (0,1))$ satisfying $v_{|\Sigma\times(0,1)}=0$,  we have
\begin{align*}
&C\|v\|_{H^1(\Omega\times(0,1))}\le \varepsilon^{1/8} \|v\|_{H^2(\Omega\times(0,1))}
\\
&\hskip 3cm +e^{e^{c/\varepsilon}}\left(\|(\Delta_g+\partial_t^2-q)v\|_{L^2(\Omega\times(0,1))}+\|\partial_{\nu_g} v\|_{L^2(\Sigma\times(0,1))}\right).
\end{align*}

Henceforth, $C=C(\zeta,\Sigma)>0$ will denote a generic constant. Since for any $\lambda\ge 1$, $v:=e^{\sqrt{\lambda}t}u$ with $u\in H^2(\Omega)\cap H^1_{0,\Sigma}(\Omega)$ satisfies
\[
C\|v\|_{H^1(\Omega\times(0,1))}\ge \left(\lambda^{-1/2}e^{\sqrt{\lambda}}\sinh\sqrt{\lambda}\right)^{1/2}\|u\|_{H^1(\Omega)},
\]
\[
C\|v\|_{H^2(\Omega\times(0,1))}\le \lambda\left(\lambda^{-1/2} e^{\sqrt{\lambda}}\sinh\sqrt{\lambda}\right)^{1/2}\|u\|_{H^2(\Omega)}
\]
and
\begin{align*}
&C\left(\|(\Delta_g+\partial_t^2-q)v\|_{L^2(\Omega\times(0,1))}+\|\partial_{\nu_g} v\|_{L^2(\Sigma\times(0,1))}\right)
\\
&\hskip 1cm \le \left(\lambda^{-1/2}e^{\sqrt{\lambda}}\sinh\sqrt{\lambda}\right)^{1/2}\left(\|(\Delta_g-q+\lambda)u\|_{L^2(\Omega)}+\|\partial_{\nu_g} u\|_{L^2(\Sigma)}\right),
\end{align*}
By combining the last three inequalities, we obtain the expected one.
\end{proof}

As a consequence of Theorem \ref{Thm2.3_CAC_2023} we have the following corollary.

\begin{cor}\label{Cor2.4_CAC_2023}
Let $1<\delta<2$. There exist constants $C=C(\zeta,\delta,\Sigma)>0$ and  $c=c(\zeta,\delta, \Sigma)>0$ such that for all $\lambda \ge 1$, $q\in L^\infty(\Omega)$ satisfying $\|q\|_{L^\infty(\Omega)}\le\kappa$, $0<\varepsilon<1$ and $u\in H^2(\Omega)\cap H^1_{0,\Sigma}(\Omega)$ we have
\begin{align*}
&C\|u\|_{H^\delta(\Omega)}\le\lambda^{2-\delta}\varepsilon^{(2-\delta)/8} \|u\|_{H^2(\Omega)}
\\
&\hskip 3.5cm +\lambda^{-(\delta-1)}e^{e^{c/\varepsilon}}\left(\|(\Delta_g-q+\lambda)u\|_{L^2(\Omega)}+\|\partial_{\nu_g} u\|_{L^2(\Sigma)}\right).
\end{align*}
\end{cor}
\begin{proof}
As $H^\delta(\Omega)$ is an interpolation space between $H^1(\Omega)$ and $H^2(\Omega)$, we have
\[
c_0\|u\|_{H^\delta(\Omega)}\le \|u\|_{H^1(\Omega)}^{2-\delta}\|u\|_{H^2(\Omega)}^{\delta-1},
\]
where $c_0=c_0(n,\Omega,g,\delta)>0$ is a generic constant. Pick $\rho>0$ arbitrarily. Young's inequality yields
\[
c_0\|u\|_{H^\delta(\Omega)}\le \lambda^{-(\delta-1)}\rho \|u\|_{H^1(\Omega)}+\lambda^{2-\delta}\rho^{-(2-\delta)/(\delta-1)}\|u\|_{H^2(\Omega)}.
\]
Applying Theorem \ref{Thm2.3_CAC_2023}, we obtain
\begin{align*}
&C\|u\|_{H^\delta(\Omega)}\le \lambda^{2-\delta}\left(\varepsilon^{1/8}\rho+\rho^{-(2-\delta)/(\delta-1)}\right)\|u\|_{H^2(\Omega)}
\\
&\hskip 3.5cm +\lambda^{-(\delta-1)}\rho e^{e^{c/\varepsilon}}\left(\|(\Delta_g-q+\lambda)u\|_{L^2(\Omega)}+\|\partial_{\nu_g} u\|_{L^2(\Sigma)}\right),
\end{align*}
where $C=C(\zeta,\delta,\Sigma)>0$ and $c=c(\zeta,\Sigma)>0$ are the constants as in Theorem \ref{Thm2.3_CAC_2023}. Upon modifying the constants $C$ and $c$, the expected inequality follows by taking $\rho=\varepsilon^{-(\delta-1)/8}$.
\end{proof}

The following corollary will be used hereinafter.

\begin{cor}\label{Cor2.5_CAC_2023}
Let $\chi$ be the characteristic function of $\Omega_0$ and $1<\delta<2$. There exist constants $C=C(\zeta,\Omega_0,\delta,\Gamma)>0$ and $c=c(\zeta,\Omega_0,\delta,\Gamma)>0$ such that for all $\lambda\in\rho(A_{q_0})$ satisfying $\lambda\ge 1$, $q\in \mathscr{Q}_\lambda$, $0<\varepsilon<1$, $f\in L^2(\Omega)$ and $u=R_q(\lambda)(\chi f)$ we have
\[
C\|u\|_{H^\delta(\Omega_1)}\le\lambda^{3-\delta}\mathbf{e}_\lambda\varepsilon^{(2-\delta)/8} \|f\|_{L^2(\Omega_0)}+\lambda^{-(\delta-1)}e^{e^{c/\varepsilon}}\|\partial_{\nu_g} u\|_{L^2(\Gamma)}.
\]
\end{cor}
\begin{proof}
Applying Corollary \ref{Cor2.4_CAC_2023} to $u=R_q(\lambda)(\chi f)\in H^2(\Omega)\cap H^1_0(\Omega)$ by replacing  $\Omega$ by $\Omega_1$ and $\Sigma$ by $\Gamma$ yields
\[
C\|u\|_{H^\delta(\Omega_1)}\le\lambda^{2-\delta}\varepsilon^{(2-\delta)/8} \|u\|_{H^2(\Omega_1)}+\lambda^{-(\delta-1)}e^{e^{c/\varepsilon}}\|\partial_{\nu_g} u\|_{L^2(\Gamma)},
\]
where $C=C(\zeta,\Omega_0,\delta,\Gamma)>0$ and $c=c(\zeta,\Omega_0,\delta,\Gamma)>0$ are the constants as in Corollary \ref{Cor2.4_CAC_2023}. Since we have $\|u\|_{H^2(\Omega_1)}\le c_1\lambda\mathbf{e}_\lambda\| f\|_{L^2(\Omega_0)}$ by Lemma \ref{Lem2.1_CAC_2023} (i), where $c_1=c_1(\zeta)>0$ is a constant, the expected inequality follows by combining these inequalities and modifying the constant $C$.
\end{proof}

\section{Quantitative Runge approximation}\label{section3}

For $\lambda\ge 0$ and $q\in L^\infty(\Omega)$, define
\[
\mathscr{S}_{q,\lambda}^0:=\left\{u\in H^2(\Omega_0);\; (-\Delta_g+q-\lambda)u=0\; \text{in}\; \Omega_0\right\}
\]
and
\[
\mathscr{S}_{q,\lambda}^1:=\left\{u\in H^2(\Omega);\; (-\Delta_g+q-\lambda)u=0\; \text{in}\; \Omega,\quad u_{|\partial\Omega}\in H^{3/2}_\Gamma(\partial\Omega)\right\}.
\]

Recall that $\zeta=(n,\Omega,g,\kappa)$.

\begin{thm}\label{Runge_approx}
Let $3/2<\delta<2$. There exist constants $C=C(\zeta,\Omega_0,\delta,\Gamma)>0$, $C'=C'(\zeta)>0$ and $c=c(\zeta,\Omega_0,\delta,\Gamma)>0$ such that for all $\lambda\in\rho(A_{q_0})$ satisfying $\lambda\ge 1$, $q\in\mathscr{Q}_\lambda$, $0<\varepsilon<1$ and $u\in \mathscr{S}_{q,\lambda}^0$ there exists $v\in \mathscr{S}_{q,\lambda}^1$ for which the following inequalities hold
\begin{align*}
C\|u-v_{|\Omega_0}\|_{L^2(\Omega_0)}&\le \lambda^{3-\delta}\mathbf{e}_\lambda \varepsilon^{(2-\delta)/8}\|u\|_{H^2(\Omega_0)},
\\
C'\|v\|_{H^2(\Omega)}&\le \lambda^2\mathbf{e}_\lambda e^{e^{c/\varepsilon}}\|u\|_{L^2(\Omega_0)}.
\end{align*}
\end{thm}

\begin{proof}
Let $\Gamma_0\subset\mbox{Int}(\Gamma)$ be a nonempty closed subset and $\psi\in C^\infty (\overline{\Omega};[0,1])$ such that $\psi_{|\Gamma_0}=1$ and $\mbox{supp}(\psi_{|\partial\Omega})\subset\mbox{Int}(\Gamma)$. Such a function $\psi$ exists. Indeed, let $U\subset\mathbb{R}^n$ be an open neighborhood of $\Gamma_0$ chosen so that $U\cap\partial\Omega\Subset\mbox{Int}(\Gamma)$. Then, it suffices to take $\psi:=\widetilde{\psi}_{|\overline{\Omega}}$, where $\widetilde{\psi}\in C_0^\infty(\mathbb{R}^n;[0,1])$ satisfying $\mbox{supp}(\widetilde{\psi})\subset U$ and $\widetilde{\psi}=1$ in a sufficiently small neighborhood of $\Gamma_0$.

Let $E:=H^{3/2}(\partial\Omega)$ and consider the closed subspace of $E$ given by
\[
F:=\{f\in E;\; \psi f=0\; \text{in}\; \partial\Omega\}.
\]
Note that $E=F\oplus F^\perp$ holds, where $F^\perp$ is the orthogonal complement with the norm of $E$.

Fix $\lambda\in\rho(A_{q_0})$ satisfying $\lambda\ge 1$ and $q\in\mathscr{Q}_\lambda$. Let $H$ denotes the closure of $\mathscr{S}_{q,\lambda}^0$ in $L^2(\Omega_0)$. Consider the operator
\[
T\colon F^\perp\ni\varphi\mapsto u_{|\Omega_0}\in H,
\]
where $u=u_{q,\lambda}(\psi\overline{\varphi})\in H^2(\Omega)$, that is, $u$ is the unique solution to the BVP
\[
\begin{cases}(-\Delta_g+q-\lambda)u=0\quad \text{in}\; \Omega,
\\
u_{|\partial\Omega}=\psi\overline{\varphi}.
\end{cases}
\]
By a priori estimate in $H^2(\Omega)$ (or Lemma \ref{Lem2.1_CAC_2023} (ii)) and Lemma \ref{multiplication_operator} with $j=2$, we see that $T\in\mathscr{B}(F^\perp; H)$.

Let $(F^\perp)^\ast$ be the dual space of $F^\perp$, $\mathscr{J}\colon (F^\perp)^\ast\rightarrow F^\perp$ be the canonical isomorphism and $\langle\cdot,\cdot\rangle$ be the pairing between $(F^\perp)^\ast$ and $F^\perp$.

Pick $v\in H\subset L^2(\Omega_0)$, $\varphi\in F^\perp$ and let $w:=R_q(\lambda)(\chi v)$, where $\chi$ is the characteristic function of $\Omega_0$. Applying Green's formula yields
\begin{align*}
&(T\varphi, v)_{L^2(\Omega_0)}=\int_\Omega u\overline{v}\chi d\mu=\int_\Omega u(-\Delta_g+q-\lambda)\overline{w} d\mu=-\int_{\partial\Omega}\psi\overline{\varphi}\partial_{\nu_g} \overline{w}ds
\\
&\hskip 1cm =-\langle\psi\partial_{\nu_g} \overline{w},\overline{\varphi}\rangle=-(\mathscr{J}(\psi\partial_{\nu_g}\overline{w}),\overline{\varphi})_{F^\perp}=(\varphi,-\overline{\mathscr{J}(\psi\partial_{\nu_g}\overline{w})})_{F^\perp}.
\end{align*}
Therefore, the adjoint operator of $T$, is given by
\[
T^\ast\colon H\ni v\mapsto -\overline{\mathscr{J}(\psi\partial_{\nu_g}\overline{w})}\in F^\perp.
\]

Let us check that $T$ is injective. If $T\varphi=u_{|\Omega_0}=0$, then $u$ satisfies
\[
(-\Delta_g+q-\lambda)u=0\; \text{in}\; \Omega,\quad u_{|\Omega_0}=0.
\]
Hence, $u=0$ according to the unique continuation property, which implies $\psi\overline{\varphi}=0$ in $\partial\Omega$ and $\varphi\in F$. By $\varphi\in F^\perp\cap F=\{0\}$, $\varphi=0$ follows.

Next, let us show that $T$ has a dense range. It suffices to prove that $\mbox{Ker}(T^\ast)=\{0\}$. If $T^\ast v=0$, then $w=R_q(\lambda)(\chi v)$ satisfies
\[
(\psi\partial_{\nu_g} \overline{w},f_1)_{L^2(\partial\Omega)}=0,\quad f_1\in F
\]
and
\[
(\psi\partial_{\nu_g}\overline{w},f_2)_{L^2(\partial\Omega)}=\langle\psi\partial_{\nu_g}\overline{w},f_2\rangle=(-\overline{T^\ast v},f_2)_{F^\perp}=0,\quad f_2\in F^\perp,
\]
that is,
\[
\int_{\partial\Omega}\psi\partial_{\nu_g} w f ds=0,\quad f\in E\, (=F\oplus F^\perp).
\]
Hence, $\psi\partial_{\nu_g} w=0$ in $\partial\Omega$ and therefore
\[
(-\Delta_g+q-\lambda)w=0\; \text{in}\; \Omega_1,\quad w_{|\partial\Omega}=0,\; \partial_{\nu_g} w_{|\Gamma_0}=0.
\]
Then, $w=0$ in $\Omega_1$ according to the unique continuation property. Taking $u\in\mathscr{S}_{q,\lambda}^0$ and applying Green's formula, we obtain
\[
(u,v)_{L^2(\Omega_0)}=(u,(-\Delta_g+q-\lambda)w)_{L^2(\Omega_0)}=((-\Delta_g+q-\lambda)u,w)_{L^2(\Omega_0)}=0.
\]
Whence, $v\in(\mathscr{S}_{q,\lambda}^0)^\perp=H^\perp=\{0\}$ implying that $v=0$ and $\mbox{Ker}(T^\ast)=\{0\}$.

We claim that $T$ is compact. Indeed, if $(\varphi_j)\subset F^\perp$ is a bounded sequence, then $\left(u_{q,\lambda}(\psi\overline{\varphi_j})\right)\subset H^2(\Omega)$ is a bounded sequence according to a priori estimate in $H^2(\Omega)$ (or Lemma \ref{Lem2.1_CAC_2023} (ii)). Subtracting if necessary a subsequence, we assume that $\left(u_{q,\lambda}(\psi\overline{\varphi_j})\right)$ converges weakly in $H^2(\Omega)$ and strongly in $L^2(\Omega)$ to $u\in L^2(\Omega)$. Since for all $j$, $u_{q,\lambda}(\psi\overline{\varphi_j})_{|\Omega_0}\in\mathscr{S}_{q,\lambda}^0$, we conclude that $u\in H$.

Summing up, we see that $T^\ast T\colon F^\perp\rightarrow F^\perp$ is compact self-adjoint and positive definite operator. Therefore, it is diagonalizable and there exists a sequence of positive numbers $(\mu_j)$ and an orthonormal basis $(\psi_j)\subset F^\perp$ so that
\[
T^\ast T\psi_j=\mu_j\psi_j.
\]
Define $\tau_j:=\mu_j^{1/2}$ and $u_j:=\mu_j^{-1/2}T\psi_j\in H$. Then, we have
\begin{align*}(u_j,u_k)_{L^2(\Omega_0)}&=\mu_j^{-1/2}\mu_k^{-1/2}(T\psi_j, T\psi_k)_{L^2(\Omega_0)}
\\
&=\mu_j^{-1/2}\mu_k^{-1/2}(T^\ast T\psi_j, \psi_k)_{F^\perp}=\delta_{jk}.
\end{align*}
Next, let $u\in H$ so that $(u,u_j)_{L^2(\Omega)}=0$ for all $j$. Then, $(u,T\psi_j)_{L^2(\Omega_0)}=0$ for all $j$ and hence $(u,T\varphi)_{L^2(\Omega_0)}=0$ for any $\varphi\in F^\perp$ by the boundedness of $T$. As $T$ has a dense range in $H$, we derive that $(u,v)_{L^2(\Omega_0)}=0$ for all $v\in H$. Taking $v:=u$ yields $u=0$ and therefore $(u_j)$ is an orthogonal basis in $H$.

Since $T^\ast u_j=\tau_j\psi_j$, we find
\begin{equation}\label{varphi_j}
\|T^\ast u_j\|_{F^\perp}=\tau_j.
\end{equation}

Let $t>0$, $a_j:=(u,u_j)_{L^2(\Omega_0)}$, $u=\sum_j a_j u_j\in H$ and set
\[
\varphi_t:=\sum_{j\in N_t^1}\tau_j^{-1}a_j\psi_j,
\]
where $N_t^1:=\{j;\; \tau_j>t\}$. We have
\begin{equation}\label{varphi_t}
\|\varphi_t\|_{F^\perp}^2=\sum_{j\in N_t^1}\tau_j^{-2}|a_j|^2\le t^{-2}\sum_j|a_j|^2=t^{-2}\|u\|_{L^2(\Omega_0)}^2.
\end{equation}

Set $N_t^0:=\{j;\; \tau_j\le t\}$ and
\[
v_t:=\sum_{j\in N_t^0}a_j u_j.
\]
Since $u-T\varphi_t=v_t$ and $T\varphi_t\perp v_t$, we find
\[
\|v_t\|_{L^2(\Omega_0)}^2=(v_t+T\varphi_t,v_t)_{L^2(\Omega_0)}=(u,v_t)_{L^2(\Omega_0)}.
\]
Let $w_t:=R_q(\lambda)(\chi v_t)$. Green's formula then yields
\begin{align*}
\|v_t\|_{L^2(\Omega_0)}^2&=(u,(-\Delta_g+q-\lambda)w_t)_{L^2(\Omega_0)}
\\
&=-(u,\partial_{\nu_g} w_t)_{L^2(\partial\Omega_0)}+(\partial_{\nu_g} u,w_t)_{L^2(\partial\Omega_0)}
\\
&\le \|u\|_{L^2(\partial\Omega_0)}\|\partial_{\nu_g} w_t\|_{L^2(\partial\Omega_0)}+\|\partial_{\nu_g} u\|_{L^2(\partial\Omega_0)}\|w_t\|_{L^2(\partial\Omega_0)}.
\end{align*}

Now, as $3/2<\delta<2$ and the trace mapping
\[
H^\delta(\Omega_1)\ni h\mapsto (h_{|\partial\Omega_0},\partial_{\nu_g} h_{|\partial\Omega_0})\in L^2(\partial\Omega_0)\times L^2(\partial\Omega_0)
\]
is bounded, we find a constant $c_0=c_0(n,\Omega,\Omega_0,g,\delta)>0$ such that
\begin{equation}\label{trace_inequality}
c_0\|v_t\|_{L^2(\Omega_0)}^2\le \|u\|_{H^2(\Omega_0)}\|w_t\|_{H^\delta(\Omega_1)}.
\end{equation}

Pick $0<\varepsilon<1$ arbitrarily. By Corollary \ref{Cor2.5_CAC_2023}, in which $\Gamma$ is replaced by $\Gamma_0$, we have
\[
C\|w_t\|_{H^\delta(\Omega_1)}\le\lambda^{3-\delta}\mathbf{e}_\lambda\varepsilon^{(2-\delta)/8} \|v_t\|_{L^2(\Omega_0)}+\lambda^{-(\delta-1)}e^{e^{c/\varepsilon}}\|\partial_{\nu_g} w_t\|_{L^2(\Gamma_0)},
\]
that is,
\[
C\|w_t\|_{H^\delta(\Omega_1)}\le\lambda^{3-\delta}\mathbf{e}_\lambda\varepsilon^{(2-\delta)/8} \|v_t\|_{L^2(\Omega_0)}+\lambda^{-(\delta-1)}e^{e^{c/\varepsilon}}\|\psi\partial_{\nu_g} w_t\|_{L^2(\partial\Omega)},
\]
where $C=C(\zeta,\Omega_0,\delta,\Gamma)>0$ and $c=c(\zeta,\Omega_0,\delta,\Gamma)>0$ are the generic constants as in Corollary \ref{Cor2.5_CAC_2023}.

On the other hand, applying the interpolation inequality in \cite[Chapter 1, Theorem 7.7]{Lions1972I}, we have
\[
c_1\|\psi\partial_{\nu_g} w_t\|_{L^2(\partial\Omega)}\le \|\psi\partial_{\nu_g} w_t\|_{H^{-3/2}(\partial\Omega)}^{1/4}\|\psi\partial_{\nu_g} w_t\|_{H^{1/2}(\partial\Omega)}^{3/4},
\]
where $c_1=c_1(n,g,\Gamma)>0$ is a constant. On the other hand, we have
\begin{align*}
\|\psi\partial_{\nu_g} w_t\|_{H^{-3/2}(\partial\Omega)}&=\sup_{f\in E\setminus\{0\}}\frac{|(\psi\partial_{\nu_g} w_t,f)_{L^2(\partial\Omega)}|}{\|f\|_E}
\\
&=\sup_{\substack{f_1\in F\setminus\{0\},\\ f_2\in F^\perp\setminus\{0\}}}\frac{|(\psi\partial_{\nu_g} w_t,f_2)_{L^2(\partial\Omega)}|}{\sqrt{\|f_1\|_F^2+\|f_2\|_{F^\perp}^2}}
\\
&\le \sup_{f_2\in F^\perp\setminus\{0\}}\frac{|(\psi\partial_{\nu_g} w_t,f_2)_{L^2(\partial\Omega)}|}{\|f_2\|_{F^\perp}}
\\
&\hskip 3cm =\|\psi\partial_\nu w_t\|_{(F^\perp)^\ast}=\|T^\ast v_t\|_{F^\perp}.
\end{align*}
Using Lemma \ref{multiplication_operator} with $j=1$, Lemma \ref{Lem2.1_CAC_2023} (i) and taking \eqref{varphi_j} into account, we obtain
\begin{align*}
c_1\|\psi\partial_{\nu_g} w_t\|_{L^2(\partial\Omega)}&\le (\lambda \mathbf{e}_\lambda)^{3/4}\|T^\ast v_t\|_{F^\perp}^{1/4}\|v_t\|_{L^2(\Omega_0)}^{3/4}
\\
&\le (\lambda \mathbf{e}_\lambda)^{3/4}t^{1/4}\|v_t\|_{L^2(\Omega_0)}
\\
&\le \lambda^{\delta-1}(\lambda^{3-\delta}\mathbf{e}_\lambda) t^{1/4}\|v_t\|_{L^2(\Omega_0)},
\end{align*}
where $c_1=c_1(\zeta,\Gamma)>0$ is a  generic constant. In consequence, we have
\[
C\|w_t\|_{H^\delta(\Omega_1)}\le \lambda^{3-\delta}\mathbf{e}_\lambda\left(\varepsilon^{(2-\delta)/8}+t^{1/4}e^{e^{c/\varepsilon}}\right)\|v_t\|_{L^2(\Omega_0)},
\]
which implies with \eqref{trace_inequality}
\[
C\|v_t\|_{L^2(\Omega_0)}\le\lambda^{3-\delta}\mathbf{e}_\lambda\left(\varepsilon^{(2-\delta)/8}+t^{1/4}e^{e^{c/\varepsilon}}\right) \|u\|_{H^2(\Omega_0)}.
\]

Taking $t:=\varepsilon^{(2-\delta)/2} e^{-4e^{c/\varepsilon}}$, we find
\begin{equation}\label{u-v}
C\|v_t\|_{L^2(\Omega_0)}\le\lambda^{3-\delta}\mathbf{e}_\lambda\varepsilon^{(2-\delta)/8}\|u\|_{H^2(\Omega_0)}
\end{equation}
and in light of \eqref{varphi_t}, we obtain 
\begin{equation}\label{v}
\|\varphi_t\|_{F^\perp}\le e^{e^{c/\varepsilon}}\|u\|_{L^2(\Omega_0)}.
\end{equation}

Let $v:=u_{q,\lambda}(\psi\overline{\varphi_t})\in \mathscr{S}_{q,\lambda}^1$. Then, it follows from Lemma \ref{Lem2.1_CAC_2023} (ii) that $\|v\|_{H^2(\Omega)}\le c_2\lambda^2\mathbf{e}_\lambda\|\varphi_t\|_E$ with a constant $c_2=c_2(\zeta)>0$. Using $u-v_{|\Omega_0}=u-T\varphi_t=v_t$, and combining the last inequality, \eqref{u-v} and \eqref{v}, we end up getting the expected inequalities.
\end{proof}

\section{Integral inequality}\label{section4}

Recall that $\zeta=(n,\Omega,g,\kappa)$ and $\mathbf{b}_\lambda=\sqrt{2\cosh(\sqrt{\lambda}/2)}$ for all $\lambda>0$ and, for simplicity, set
\[
\mathbf{k}_\lambda:=\lambda^5\mathbf{e}_\lambda^3\mathbf{b}_\lambda,\quad \lambda\in\rho(A_{q_0}).
\]

\begin{prop}\label{integral_inequality}
Let $2<\delta<3$. There exist constants $C=C(\zeta,\Omega_0,\delta,\Gamma,\Sigma)>0$, $c=c(\zeta,\Omega_0,\delta,\Gamma)>0$ and $\theta=\theta(\zeta,\Omega_0,\Sigma)\in(0,1)$ such that for all $\lambda\in\rho(A_{q_0})$ so that $\lambda\ge 1$, $q_j\in\mathscr{Q}_\lambda$, $j=1,2$, satisfying $q_1=q_2$ in $\Omega_1$, $0<\varepsilon<1$ and $u_j\in \mathscr{S}_{q_j,\lambda}^0$, $j=1,2$, we have
\begin{align*}
&C\left|\int_{\Omega_0}(q_1-q_2)u_1 u_2 d\mu\right|\le \lambda^\delta \varepsilon^{(2-\delta)/16} \|u_1\|_{H^2(\Omega_0)}\|u_2\|_{H^2(\Omega_0)}
\\
&\hskip 4.5cm +\mathbf{k}_\lambda e^{e^{c/\varepsilon}}\|\Lambda_{q_1,\lambda}-\Lambda_{q_2,\lambda}\|^\theta\|u_1\|_{L^2(\Omega_0)}\|u_2\|_{L^2(\Omega_0)}.
\end{align*}
\end{prop}

\begin{proof}
Fix $\lambda\in\rho(A_{q_0})$ satisfying $\lambda\ge 1$. Let $q_1,q_2\in\mathscr{Q}_\lambda$ satisfying $q_1=q_2$ in $\Omega_1$. Pick $3/2<\delta<2$, $0<\varepsilon<1$ and $u_j\in\mathscr{S}_{q_j,\lambda}^0$, $j=1,2$. By Theorem \ref{Runge_approx}, there exist $v_j\in\mathscr{S}_{q_j,\lambda}^1$, $j=1,2$ such that
\begin{align}
\label{u_j-v_j}C\|u_j-{v_j}_{|\Omega_0}\|_{L^2(\Omega_0)}&\le \lambda^{3-\delta}\mathbf{e}_\lambda \varepsilon^{(2-\delta)/8}\|u_j\|_{H^2(\Omega_0)},
\\
\label{v_j}C'\|v_j\|_{H^2(\Omega)}&\le \lambda^2\mathbf{e}_\lambda e^{e^{c/\varepsilon}}\|u_j\|_{L^2(\Omega_0)},
\end{align}
where $C=C(\zeta,\Omega_0,\delta,\Gamma)>0$, $C'=C'(\zeta)>0$ and $c=c(\zeta,\Omega_0,\delta,\Gamma)>0$ are the generic constants as in Theorem \ref{Runge_approx}.

By using \eqref{u_j-v_j}, the identity
\[
u_1u_2=(u_1-v_1)u_2+(v_1-u_1)(u_2-v_2)+u_1(u_2-v_2)+v_1v_2
\]
and $0<\varepsilon<1$, we have
\begin{align}\label{(25)_CAC_2023}
&C\left|\int_{\Omega_0}(q_1-q_2)u_1 u_2d\mu\right|\le \lambda^{2(3-\delta)}\mathbf{e}_\lambda^2\varepsilon^{(2-\delta)/8}\|u_1\|_{H^2(\Omega_0)}\|u_2\|_{H^2(\Omega_0)}
\\
&\hskip 7cm +\left|\int_{\Omega_0}(q_1-q_2)v_1 v_2 d\mu\right|.\notag
\end{align}

Let $\widetilde{v}_2\in H^2(\Omega)$ be the unique solution of the BVP
\[
\begin{cases}
(-\Delta_g+q_1-\lambda)\widetilde{v}_2=0\quad \text{in}\; \Omega,
\\
\widetilde{v}_2{}_{|\partial\Omega}=v_2{}_{|\partial\Omega}
\end{cases}
\]
and set $v:=v_2-\widetilde{v}_2 =R_{q_1}(\lambda)((q_1-q_2)v_2)$. Pick $\psi\in C_0^\infty(\Omega;[0,1])$ so that $\psi=1$ in an open neighborhood of $\overline{\Omega_0}$. Then, $w:=\psi v$ satisfies
\[
(-\Delta_g+q_1-\lambda)w=\psi(q_1-q_2)v_2-[\Delta_g,\psi]v\quad \text{in}\; \Omega.
\]
Applying Green's formula, we have
\[
\int_\Omega\psi(q_1-q_2)v_2v_1d\mu=\int_\Omega[\Delta_g,\psi]vv_1d\mu.
\]
Since there exists an open set $U\Subset\Omega_1$ such that $\mbox{supp}([\Delta_g,\psi]v)\subset U$ and $q_1=q_2$ in $\Omega_1$, we get
\[
\left|\int_{\Omega_0}(q_1-q_2)v_2v_1d\mu\right|=\left|\int_\Omega[\Delta_g,\psi]vv_1d\mu\right|\le c_0\|v\|_{H^1(U)}\|v_1\|_{L^2(\Omega)},
\]
where $c_0=c_0(n,\Omega,g,\Omega_0)>0$ is a constant. This inequality, combined with Theorem \ref{Thm2.2_CAC_2023}, in which $\Omega$ and $q$ are replaced by $\Omega_1$ and $q_1$, Lemma \ref{Lem2.1_CAC_2023} (i) and \eqref{v_j} yields
\begin{align*}
\left|\int_{\Omega_0}(q_1-q_2)v_2v_1d\mu\right|&\le C\mathbf{b}_\lambda\left(\sqrt{\lambda}\rho^p\|v\|_{H^1(\Omega_1)}+\rho^{-1}\|\partial_{\nu_g} v\|_{L^2(\Sigma)}\right)\|v_1\|_{L^2(\Omega)}
\\
&\le C\mathbf{b}_\lambda\left(\lambda\mathbf{e}_\lambda\rho^p\|v_2\|_{L^2(\Omega)}+\rho^{-1}\|\partial_{\nu_g} v\|_{L^2(\Sigma)}\right)\|v_1\|_{L^2(\Omega)}
\\
&\le C\mathbf{b}_\lambda\left(\lambda^3\mathbf{e}_\lambda^2 e^{e^{c/\varepsilon}}\rho^p\|u_2\|_{L^2(\Omega_0)}+\rho^{-1}\|\partial_{\nu_g} v\|_{L^2(\Sigma)}\right)\|v_1\|_{L^2(\Omega)},
\end{align*}
where $\rho>0$ is taken arbitrarily, $C=C(\zeta,\Omega_0,\delta,\Gamma,\Sigma)>0$ and $p=p(\zeta,\Omega_0,\Sigma)>0$ are generic constants.

On the other hand, it follows from \eqref{v_j} that
\begin{align*}
\|\partial_{\nu_g} v\|_{L^2(\Sigma)}&\le\|(\Lambda_{q_1,\lambda}-\Lambda_{q_2,\lambda})(v_2{}_{|\partial\Omega})\|_{H^{1/2}(\Sigma)}
\\
&\le\|\Lambda_{q_1,\lambda}-\Lambda_{q_2,\lambda}\|\|v_2\|_{H^{3/2}_\Gamma(\partial\Omega)}
\\
&\le C'\lambda^2\mathbf{e}_\lambda e^{e^{c/\varepsilon}}\|\Lambda_{q_1,\lambda}-\Lambda_{q_2,\lambda}\|\|u_2\|_{L^2(\Omega_0)}.
\end{align*}
Therefore, we obtain
\begin{align*}
&C\left|\int_{\Omega_0}(q_1-q_2)v_2v_1d\mu\right|
\\
&\hskip 2cm \le \lambda^5\mathbf{e}_\lambda^3\mathbf{b}_\lambda e^{e^{c/\varepsilon}}\left(\rho^p+\rho^{-1}\|\Lambda_{q_1,\lambda}-\Lambda_{q_2,\lambda}\|\right)\|u_1\|_{L^2(\Omega_0)}\|u_2\|_{L^2(\Omega_0)}.
\end{align*}

Taking in this inequality $\rho:=\|\Lambda_{q_1,\lambda}-\Lambda_{q_2,\lambda}\|^{1/(p+1)}$, we get
\[
C\left|\int_{\Omega_0}(q_1-q_2)v_2v_1d\mu\right|
\le \mathbf{k}_\lambda e^{e^{c/\varepsilon}}\|\Lambda_{q_1,\lambda}-\Lambda_{q_2,\lambda}\|^\theta\|u_1\|_{L^2(\Omega_0)}\|u_2\|_{L^2(\Omega_0)},
\]
where $\theta=\theta(\zeta,\Omega_0,\Sigma):=p/(p+1)\in(0,1)$. Putting together this inequality, \eqref{(25)_CAC_2023} and replacing $2(3-\delta)$ by $\delta$, we complete the proof.
\end{proof}

%\begin{remark}
%{\rm
%When $g=\mathfrak{c}g_0$, where $g_0$ is known metric and $\mathfrak{c}$ is an unknown conformal factor, one can proceed as \cite[Corollary 5.2]{Choulli2023a} to prove that $\mathfrak{c}$ is uniquely determined from the corresponding Dirichlet-to-Neumann map.
%}
%\end{remark}

\section{Proof of Theorem \ref{triple_log_stability}}\label{section5}

Assume that $g=\mathbf{I}$. Let $\lambda\in\rho(A_{q_0})$ satisfying $\lambda\ge 1$ and $q_j \in \mathscr{Q}_\lambda$, $j=1,2$ such that $q_1=q_2$ in $\Omega_1$. Pick $\eta\in\mathbb{R}^n$ and $\tau\ge\varpi$ arbitrarily, where the constant $\varpi=\varpi(n,\Omega_0,\kappa)\ge 1$ is as in Proposition \ref{Prop6.1_CAC_2023} with $X:=\Omega_0$. Choose $\eta_1,\eta_2\in\mathbb{R}^n\setminus\{0\}$ so that $\eta_1\perp\eta_2$, $\eta_j\perp\eta$, $j=1,2$, $|\eta_1|^2:=\tau^2+\lambda$ and
\[
|\eta_2|^2:=|\eta|^2/4+|\eta_1|^2-\lambda=|\eta|^2/4+\tau^2.
\]
Then, set
\[
\xi_1:=(\eta/2+\eta_1)+i\eta_2,\quad \xi_2:=(\eta/2-\eta_1)-i\eta_2.
\]
We check that for $j=1,2$, $\xi_j\cdot\xi_j=\lambda$, $\xi_1+\xi_2=\eta$ and
\[
(\varpi\le)\, \tau\le |\eta_2|=|\Im\xi_j|\le |\eta|/2+\tau.
\]

Set $u_j:=u_{\xi_j}$ and $w_j:=w_{\xi_j}$ for $j=1,2$, where $u_{\xi_j}$ and $w_{\xi_j}$ are as in Proposition \ref{Prop6.1_CAC_2023} with $q=q_j$ and $\xi=\xi_j$.

Let $q:=q_1-q_2$ extended by $0$ in $\mathbb{R}^n\setminus\Omega$. Since
\[
u_1 u_2=e^{-i\eta\cdot x}+\varrho,\quad \varrho:=e^{-i\eta\cdot x}(w_1+w_2+w_1 w_2),
\]
we find
\[
\hat{q}(\eta)=\int_{\Omega_0} qu_1u_2dx-\int_{\Omega_0}q\varrho dx,
\]
where $\hat{q}$ denotes the Fourier transform of $q$. From Proposition \ref{Prop6.1_CAC_2023}, we find that for any $\eta\in\mathbb{R}^n$ and $\tau\ge \varpi$,
\[
|\hat{q}(\eta)|\le \left|\int_{\Omega_0} qu_1u_2dx\right|+c_1\tau^{-1},
\]
where $c_1=c_1(n,\Omega_0,\kappa)>0$ is a constant.

On the other hand, it follows again from Proposition \ref{Prop6.1_CAC_2023} that for any $\tau\ge\varpi$,
\begin{align*}
&\|u_1\|_{L^2(\Omega_0)}\|u_2\|_{L^2(\Omega_0)}\le c_2e^{\varkappa(|\eta|+2\tau)},
\\
&\|u_1\|_{H^2(\Omega_0)}\|u_2\|_{H^2(\Omega_0)}\le c_2\lambda^2 e^{\varkappa(|\eta|+2\tau)},
\end{align*}
where  the constants $\varkappa=\varkappa(n,\Omega_0,\kappa)>0$ and $c_2=c_2(n,\Omega_0,\kappa)>0$ are  as in Proposition \ref{Prop6.1_CAC_2023}.

Combining these inequalities and Proposition \ref{integral_inequality} with $\delta:=5/2$, we obtain
\[
C|\hat{q}(\eta)|\le\tau^{-1}+\lambda^{9/2} \varepsilon^{1/32} e^{\varkappa(|\eta|+2\tau)}+\mathbf{k}_\lambda e^{e^{c/\varepsilon}} e^{\varkappa(|\eta|+2\tau)}\|\Lambda_{q_1,\lambda}-\Lambda_{q_2,\lambda}\|^\theta.
\]
Here and henceforth, $\mathbf{k}_\lambda=\lambda^5\mathbf{e}_\lambda^3\mathbf{b}_\lambda$ is as in Proposition \ref{integral_inequality},  and $C=C(\varsigma,\Gamma,\Sigma)>0$, $c=c(\varsigma,\Gamma)>0$ and $\theta=\theta(\varsigma,\Sigma)\in(0,1)$ are generic constants, where $\varsigma=(n,\Omega,\Omega_0,\kappa)$.

Pick $s>0$ and set $\mathfrak{C}:=\|\Lambda_{q_1,\lambda}-\Lambda_{q_2,\lambda}\|^\theta$. Then, we have for $|\eta|\le s$
\[
C\mathbf{k}_\lambda^{-1}|\hat{q}(\eta)|\le\tau^{-1}+\varepsilon^{1/32} e^{\varkappa(s+2\tau)}+e^{e^{c/\varepsilon}} e^{\varkappa(s+2\tau)}\mathfrak{C}.
\]
Therefore
\[
C\mathbf{k}_\lambda^{-2}\int_{|\eta|\le s}|\hat{q}(\eta)|^2d\eta
\le s^n\tau^{-2}+s^n\varepsilon^{1/16} e^{\varkappa(s+2\tau)}+s^ne^{e^{c/\varepsilon}} e^{\varkappa(s+2\tau)}\mathfrak{C}^2.
\]

On the other hand, we verify that
\[
\int_{|\eta|> s}(1+|\eta|^2)^{-1}|\hat{q}(\eta)|^2d\eta\le s^{-2}\|q\|_{L^2(\mathbb{R}^n)}^2\le c_3s^{-2},
\]
where $c_3=c_3(\Omega,\kappa)>0$ is a constant.

Combining these inequalities, we obtain
\[
C\mathbf{k}_\lambda^{-2}\|q\|_{H^{-1}(\mathbb{R}^n)}^2
\le s^{-2}+s^n\tau^{-2}+s^n\varepsilon^{1/16} e^{\varkappa(s+2\tau)}+s^ne^{e^{c/\varepsilon}} e^{\varkappa(s+2\tau)}\mathfrak{C}^2.
\]
Taking $s=\tau^{2/(n+2)}\le\tau$, we obtain
\[
C\mathbf{k}_\lambda^{-2}\|q\|_{H^{-1}(\mathbb{R}^n)}^2
\le \tau^{-4/(n+2)}+\varepsilon^{1/16} e^{\varkappa\tau}+e^{e^{c/\varepsilon}} e^{\varkappa\tau}\mathfrak{C}^2.
\]
Let $\tau_0>0$ be sufficiently large to satisfy $\tau_0^{-4/(n+2)}e^{-\kappa\tau_0}<1$. In this case, for each $\tau\ge \tau_0$, there exists $0<\varepsilon<1$ such that $\varepsilon=\tau^{-64/(n+2)}e^{-16\varkappa\tau}$. Hence, for all $\tau\ge \tau_1:=\max(\varpi,\tau_0)$, we have
\[
C\mathbf{k}_\lambda^{-2}\|q\|_{H^{-1}(\mathbb{R}^n)}^2
\le \tau^{-4/(n+2)}+\mathfrak{e}(\varkappa\tau)\mathfrak{C}^2,
\]
where $\varkappa=\varkappa(\varsigma,\Gamma)>0$ is a generic constant, and $\mathfrak{e}(c)=e^{e^{e^c}}$, $c>0$. Hence, we have

\[
C\mathbf{k}_\lambda^{-1}\|q\|_{H^{-1}(\mathbb{R}^n)}
\le \tau^{-2/(n+2)}+\mathfrak{e}(\varkappa\tau)\mathfrak{C}.
\]
By replacing $\tau$ by $\tau/\tau_1$,  we see that the above inequality holds for all $\tau\ge 1$.

When $0<\mathfrak{e}(\varkappa)\mathfrak{C}<1$, we choose $\tau\ge 1$ so that
\[
\left(\mathfrak{e}(\varkappa)\mathfrak{C}\le \right)\, 
\mathfrak{e}(\varkappa\tau)\mathfrak{C}=\tau^{-2/(n+2)}\, (\le 1),
\]
that is,
\[
\tau^{2/(n+2)}\mathfrak{e}(\varkappa\tau)=1/\mathfrak{C}\, \left(>\mathfrak{e}(\varkappa)\right).
\]
Hence, there exists $\varkappa'>\varkappa$ independent of $\tau\ge 1$ such that
\[
\varkappa'\tau\ge L (1/\mathfrak{C}),
\]
where $L(r)=\log \log \log (r)$, $r>\mathfrak{e}(c)$, $c>0$. Therefore, we have
\begin{align*}
C\mathbf{k}_\lambda^{-1}\|q\|_{H^{-1}(\mathbb{R}^n)}&\le 2\tau^{-2/(n+2)}
\\
&\le 2{\varkappa'}^{2/(n+2)}L(1/\mathfrak{C})^{-2/(n+2)}.
\end{align*}

When $1\le \mathfrak{e}(\varkappa)\mathfrak{C}$, we have
\[
C\mathbf{k}_\lambda^{-1}\|q\|_{H^{-1}(\mathbb{R}^n)}
\le \left(\mathfrak{e}(\varkappa)\tau^{-2/(n+2)}+\mathfrak{e}(\varkappa\tau)\right)\mathfrak{C}.
\]
By taking $\tau=1$, we obtain 
\[
C\mathbf{k}_\lambda^{-1}\|q\|_{H^{-1}(\mathbb{R}^n)}
\le 2\mathfrak{e}(\varkappa)\mathfrak{C}.
\]
This completes the proof.

\section{The interior impedance problem}\label{section6}

\subsection{Statement of the result}

Let $n\ge 3$ be an integer. In this section, we assume $\Omega\subset\mathbb{R}^n$ is a $C^\infty$ bounded domain and $g=\mathbf{I}$. Fix $a\in C^\infty(\overline{\Omega};\mathbb{R})$ such that $a>0$ and consider the BVP
\begin{equation}\label{BVP_2}
\begin{cases}
(-\Delta+q-\lambda)u=f\quad &\text{in}\; \Omega,
\\
(\partial_\nu\mp ia\sqrt{\lambda})u=\varphi\quad &\text{on}\; \partial\Omega,
\end{cases}
\end{equation}
where $\lambda>0$, $q\in L^\infty(\Omega;\mathbb{R})$, $f\in L^2(\Omega)$ and $\varphi\in L^2(\partial\Omega)$. We proceed as in the proof of \cite[Proposition A.1]{Krupchyk2019} to prove that the BVP \eqref{BVP_2} has a unique solution $u=u_{q,\lambda}^\pm (f,\varphi)\in H^1(\Omega)$. 

Fix $\kappa>0$ arbitrarily and set
\[
\mathscr{Q}:=\left\{q\in L^\infty(\Omega;\mathbb{R});\; \|q\|_{L^\infty(\Omega)}\le\kappa\right\}.
\]
Let $\lambda_0=\lambda_0(n,\Omega,\kappa,a)>0$ be the constant as in Lemma \ref{auxiliary_estimates}. According to Lemma \ref{auxiliary_estimates}, if $\lambda\ge \lambda_0$, $q\in\mathscr{Q}$ and $\varphi\in H^{1/2}(\partial\Omega)$, then $u\in H^2(\Omega)$.

We associate to $\lambda\ge\lambda_0$ and $q\in \mathscr{Q}$ the partial Robin-to-Dirichlet operator $\mathscr{N}_{q,\lambda}$ given by
\[
\mathscr{N}_{q,\lambda}\colon H^{1/2}_\Gamma(\partial\Omega)\ni \varphi\mapsto u_{q,\lambda}^\pm (0,\varphi)_{|\Sigma}\in H^{3/2}(\Sigma),
\]
where $H^{1/2}_\Gamma(\partial\Omega)$  is the closed subspace of $H^{1/2}(\partial\Omega)$ given by
\[
H^{1/2}_\Gamma(\partial\Omega):=\left\{\varphi\in H^{1/2}(\partial\Omega);\; \mbox{supp}(\varphi)\subset \Gamma\right\}
\]
and endowed with the  norm of $H^{1/2}(\partial\Omega)$, and
\[
H^{3/2}(\Sigma):=\{\varphi=h_{|\Sigma};\; h\in H^{3/2}(\partial\Omega)\}
\]
endowed with the quotient norm
\[
\|\varphi\|_{H^{3/2}(\Sigma)}:=\inf\{\|h\|_{H^{3/2}(\partial\Omega)};\; h_{|\Sigma}=\varphi\}.
\]

It follows from \eqref{(39)_CAC_2023} that $\mathscr{N}_{q,\lambda}\in \mathscr{B}(H^{1/2}_\Gamma(\partial\Omega); H^{3/2}(\Sigma))$.

Recall that
\[
\mathbf{b}_\lambda=\sqrt{2\cosh(\sqrt{\lambda}/2)},\quad \lambda>0,
\]
and $\Phi_c\colon(0,\infty)\rightarrow\mathbb{R}$ for $c>0$ is given by
\[
\Phi_{c} (r)=r^{-1}\chi_{\left]0,\mathfrak{e}(c)\right]}(r)+L(r)^{-2/(n+2)}\chi_{\left]\mathfrak{e}(c),\infty\right[}(r),
\]
where $\mathfrak{e}(c)=e^{e^{e^c}}$ and $L(r)=\log\log\log r$, $r>\mathfrak{e}(c)$.

The second main result of this paper is the following theorem.

\begin{thm}\label{triple_log_stability_2}
Assume that $g=\mathbf{I}$ and let $\varsigma=(n,\Omega,\Omega_0,\kappa)$. There exist constants $C=C(\varsigma,a,\Gamma,\Sigma)>0$, $c=c(\varsigma,\Gamma)>0$ and $\theta=\theta(\varsigma,\Sigma)\in(0,1)$ such that for all $\lambda\ge\lambda_0$ and $q_j\in\mathscr{Q}$, $j=1,2$ satisfying $q_1=q_2$ in $\Omega_1$ we have
\[
C\|q_1-q_2\|_{H^{-1}(\Omega)}
\le \lambda^6\mathbf{b}_\lambda \Phi_{c}\left(\|\mathscr{N}_{q_1,\lambda}-\mathscr{N}_{q_2,\lambda}\|^{-\theta}\right).
\]
\end{thm}

A double logarithmic stability inequality has already been obtained by the first author (\cite[Theorem 7.5]{Choulli2023a}) in the case where $n=3$. We also mention that, still when $n=3$, the first author proved in \cite[Theorem 7.2]{Choulli2023a} a single logarithmic stability inequality for the partial Robin-to-Dirichlet map given by
\[
H^{1/2}(\partial\Omega)\ni \varphi\mapsto u_{q,\lambda}^\pm (0,\varphi)_{|\Sigma}\in H^{3/2}(\Sigma).
\]

\subsection{Global quantitative uniqueness of continuation}

Define
\[
\mathscr{H}^\pm:=\left\{u\in H^2(\Omega);\; (\partial_\nu\mp ia\sqrt{\lambda})u_{|\Sigma}=0\right\}.
\]

Fix $\kappa>0$ arbitrarily. Set $\zeta_0:=(n,\Omega,\kappa)$ and recall that the constant $\lambda_0=\lambda_0(\zeta_0,a)>0$ is as in Lemma \ref{auxiliary_estimates}.

\begin{thm}\label{Thm2.3_CAC_2023_2}
There exist constants $C=C(\zeta_0,a,\Sigma)>0$ and $c=c(\zeta_0,\Sigma)>0$ such that for all $\lambda\ge\lambda_0$, $q\in L^\infty(\Omega)$ satisfying $\|q\|_{L^\infty(\Omega)}\le\kappa$, $0<\varepsilon<1$ and $u\in \mathscr{H}^\pm$ we have
\[
C\|u\|_{H^1(\Omega)}\le\lambda\varepsilon^{1/8} \|u\|_{H^2(\Omega)}+\sqrt{\lambda}e^{e^{c/\varepsilon}}\left(\|(\Delta-q+\lambda)u\|_{L^2(\Omega)}+\|u\|_{H^1(\Sigma)}\right).
\]
\end{thm}
\begin{proof}
According to \cite[Theorem 1]{Choulli2020} with $j=1$ and $s=1/4$, there exist $C=C(\zeta_0,a,\Sigma)>0$, $c=c(\zeta_0,\Sigma)>0$, for any  $0<\varepsilon<1$ and $v\in H^2(\Omega\times (0,1))$, we have
\begin{align*}
&C\|v\|_{H^1(\Omega\times(0,1))}\le \varepsilon^{1/8} \|v\|_{H^2(\Omega\times(0,1))}+e^{e^{c/\varepsilon}}\left(\|(\Delta+\partial_t^2-q)v\|_{L^2(\Omega\times(0,1))}\right.
\\
&\hskip 6cm \left.+\|\partial_\nu v\|_{L^2(\Sigma\times(0,1))}+\|v\|_{H^1(\Sigma\times(0,1))}\right).
\end{align*}

Henceforth, $C=C(\zeta_0,a,\Sigma)>0$ will denote a generic constant. Since for any $\lambda\ge\lambda_0$, $v:=e^{\sqrt{\lambda}t}u$ with $u\in \mathscr{H}^\pm$ satisfies $\partial_\nu v=\pm ia\sqrt{\lambda}e^{\sqrt{\lambda}t}u$ on $\Sigma\times(0,1)$,
\[
C\|v\|_{H^1(\Omega\times(0,1))}\ge \left(\lambda^{-1/2}e^{\sqrt{\lambda}}\sinh\sqrt{\lambda}\right)^{1/2}\|u\|_{H^1(\Omega)},
\]
\[
C\|v\|_{H^2(\Omega\times(0,1))}\le \lambda\left(\lambda^{-1/2} e^{\sqrt{\lambda}}\sinh\sqrt{\lambda}\right)^{1/2}\|u\|_{H^2(\Omega)}
\]
and
\begin{align*}
&C\left(\|(\Delta+\partial_t^2-q)v\|_{L^2(\Omega\times(0,1))}+\|\partial_\nu v\|_{L^2(\Sigma\times(0,1))}+\|v\|_{H^1(\Sigma\times(0,1))}\right)
\\
&\hskip 1cm \le \sqrt{\lambda}\left(\lambda^{-1/2}e^{\sqrt{\lambda}}\sinh\sqrt{\lambda}\right)^{1/2}\left(\|(\Delta-q+\lambda)u\|_{L^2(\Omega)}+\|u\|_{H^1(\Sigma)}\right),
\end{align*}
the expected inequality follows.
\end{proof}

As a consequence of Theorem \ref{Thm2.3_CAC_2023_2} we have the following corollary. Recall that $\varsigma=(n,\Omega,\Omega_0,\kappa)$.

\begin{cor}\label{Cor2.5_CAC_2023_2}
Let $\chi$ be the characteristic function of $\Omega_0$. There exist $C=C(\varsigma,a,\Gamma)>0$ and $c=c(\varsigma,\Gamma)>0$ such that  for all $\lambda\ge\lambda_0$, $q\in \mathscr{Q}$, $0<\varepsilon<1$, $f\in L^2(\Omega)$ and $u=u_{q,\lambda}^\pm(\chi f,0)$, we have
\[
C\|u\|_{H^1(\Omega_1)}\le\lambda^{3/2}\varepsilon^{1/8} \|f\|_{L^2(\Omega_0)}+\sqrt{\lambda}e^{e^{c/\varepsilon}}\|u\|_{H^1(\Gamma)}.
\]
\end{cor}

\begin{proof}
Applying Theorem \ref{Thm2.3_CAC_2023_2} to $u=u_{q,\lambda}^\pm(\chi f,0)\in \mathscr{H}^\pm$ by replacing $\Omega$ with $\Omega_1$ and $\Sigma$ with $\Gamma$ yields
\begin{equation}\label{x1}
C\|u\|_{H^1(\Omega_1)}\le\lambda\varepsilon^{1/8} \|u\|_{H^2(\Omega_1)}+\sqrt{\lambda}e^{e^{c/\varepsilon}}\|u\|_{H^1(\Gamma)},
\end{equation}
where $C=C(\varsigma,a,\Gamma)>0$ and $c=c(\varsigma,\Gamma)>0$ are the constants as in Theorem \ref{Thm2.3_CAC_2023_2}. 

On the other hand, we have by \eqref{(39)_CAC_2023}  
\begin{equation}\label{x2}
\|u\|_{H^2(\Omega)}\le c_1\sqrt{\lambda}\|f\|_{L^2(\Omega_0)}, 
\end{equation}
where $c_1=c_1(\zeta_0,a)>0$ is a constant, 

Upon modifying the constant $C$ in \eqref{x1}, the expected inequality then follows by combining \eqref{x1} and \eqref{x2}.
\end{proof}

\subsection{Quantitative Runge approximation}

For $\lambda\ge 0$ and $q\in L^\infty(\Omega)$, recall that
\[
\mathscr{S}_{q,\lambda}^0=\left\{u\in H^2(\Omega_0);\; (-\Delta+q-\lambda)u=0\quad \text{in}\; \Omega_0\right\}
\]
and set
\[
\mathscr{S}_{q,\lambda}^\pm:=\left\{u\in H^2(\Omega);\; (-\Delta+q-\lambda)u=0\; \text{in}\; \Omega,\quad (\partial_\nu\mp ia\sqrt{\lambda})u_{|\partial\Omega}\in H^{1/2}_\Gamma(\partial\Omega)\right\}.
\]

Fix $\kappa>0$ arbitrarily and recall that $\zeta_0=(n,\Omega,\kappa)$, $\varsigma=(n,\Omega,\Omega_0,\kappa)$ and the constant $\lambda_0=\lambda_0(\zeta_0,a)>0$ is  as in Lemma \ref{auxiliary_estimates}.

\begin{thm}\label{Runge_approx_2}
There exist constants $C=C(\varsigma,a,\Gamma)>0$, $C'=C'(\zeta_0,a)>0$ and $c=c(\varsigma,\Gamma)>0$ such that for all $\lambda\ge\lambda_0$, $q\in\mathscr{Q}$, $0<\varepsilon<1$ and $u\in \mathscr{S}_{q,\lambda}^0$, there exists $v\in \mathscr{S}_{q,\lambda}^\pm$ for which the following inequalities hold
\begin{align*}
C\|u-v_{|\Omega_0}\|_{L^2(\Omega_0)}&\le \lambda^2\varepsilon^{1/8}\|u\|_{H^2(\Omega_0)},
\\
C'\|v\|_{H^2(\Omega)}&\le \sqrt{\lambda}e^{e^{c/\varepsilon}}\|u\|_{L^2(\Omega_0)}.
\end{align*}
\end{thm}

\begin{proof}
Let $\Gamma_0\subset\mbox{Int}(\Gamma)$ be a nonempty closed subset and $\psi\in C^\infty(\overline{\Omega};[0,1])$ such that $\psi_{|\Gamma_0}=1$ and $\mbox{supp}(\psi_{|\partial\Omega})\subset\mbox{Int}(\Gamma)$. In the proof of Theorem \ref{Runge_approx}, the existence of such a function $\psi$ is guaranteed.

Let $\mathcal{E}:=H^{1/2}(\partial\Omega)$ and consider the closed subspace of $\mathcal{E}$ given by
\[
\mathcal{F}:=\{f\in \mathcal{E};\; \psi f=0\; \text{in}\; \partial\Omega\}.
\]
Note that $\mathcal{E}=\mathcal{F}\oplus \mathcal{F}^\perp$ holds, where $\mathcal{F}^\perp$ is the orthogonal complement with the norm of $\mathcal{E}$.

Fix $\lambda\ge\lambda_0$ and $q\in\mathscr{L}$. Let $H$ denotes the closure of $\mathscr{S}_{q,\lambda}^0$ in $L^2(\Omega_0)$ and consider the operator
\[
\mathcal{T}\colon \mathcal{F}^\perp\ni\varphi\mapsto u_{q,\lambda}^\pm(0,\psi\overline{\varphi}){}_{|\Omega_0}\in H,
\]
where $u=u_{q,\lambda}^\pm(0,\psi\overline{\varphi})\in H^2(\Omega)$, that is, $u$ is the unique solution to the BVP
\[
\begin{cases}(-\Delta+q-\lambda)u=0\quad &\text{in}\; \Omega,
\\
\partial_\nu u\mp ia\sqrt{\lambda}u=\psi\overline{\varphi}\quad &\text{on}\; \partial\Omega.
\end{cases}
\]
By \eqref{(39)_CAC_2023} and Lemma \ref{multiplication_operator} with $j=1$, we see that $\mathcal{T}\in\mathscr{B}(\mathcal{F}^\perp; H)$.

Let $(\mathcal{F}^\perp)^\ast$ be the dual space of $\mathcal{F}^\perp$, $\mathfrak{J}\colon (\mathcal{F}^\perp)^\ast\rightarrow \mathcal{F}^\perp$ be the canonical isomorphism and $\langle\cdot,\cdot\rangle$ be the pairing between $(\mathcal{F}^\perp)^\ast$ and $\mathcal{F}^\perp$.

Pick $v\in H\subset L^2(\Omega_0)$, $\varphi\in \mathcal{F}^\perp$ and let $w:=u_{q,\lambda}^\mp (\chi v,0)$, where $\chi$ is the characteristic function of $\Omega_0$. Applying Green's formula yields
\begin{align*}
(\mathcal{T}\varphi, v)_{L^2(\Omega_0)}&=\int_\Omega u\overline{v}\chi d\mu=\int_\Omega u(-\Delta+q-\lambda)\overline{w} d\mu
\\
&=-\int_{\partial\Omega}u\partial_\nu \overline{w}ds+\int_{\partial\Omega}\partial_\nu u \overline{w}ds=\int_{\partial\Omega}\psi\overline{\varphi}\overline{w}ds
\\
&=\langle \psi\overline{w},\overline{\varphi}\rangle=(\mathfrak{J}(\psi\overline{w}),\overline{\varphi})_{\mathcal{F}^\perp}=(\varphi,\overline{\mathfrak{J}(\psi\overline{w})})_{\mathcal{F}^\perp}.
\end{align*}
Then, the adjoint operator of $\mathcal{T}$ is given by
\[
\mathcal{T}^\ast\colon H\ni v\mapsto \overline{\mathfrak{J}(\psi\overline{w})}\in \mathcal{F}^\perp.
\]

Let verify that $\mathcal{T}$ is injective. If $\mathcal{T}\varphi=u_{|\Omega_0}=0$, then $u$ satisfies
\[
(-\Delta+q-\lambda)u=0\; \text{in}\; \Omega,\quad u_{|\Omega_0}=0.
\]
Hence, $u=0$ according to the unique continuation property, which implies $\psi\overline{\varphi}=0$ in $\partial\Omega$ and $\varphi\in \mathcal{F}$. By $\varphi\in \mathcal{F}^\perp\cap \mathcal{F}=\{0\}$, $\varphi=0$ follows.

Next, we show that $\mathcal{T}$ has a dense range. For this, it suffices to prove that $\mbox{Ker}(\mathcal{T}^\ast)=\{0\}$. If $\mathcal{T}^\ast v=0$, then $w=u_{q,\lambda}^\mp (\chi v,0)$ satisfies
\[
(\psi \overline{w},f_1)_{L^2(\partial\Omega)}=0,\quad f_1\in \mathcal{F}
\]
and
\[
(\psi\overline{w},f_2)_{L^2(\partial\Omega)}=\langle\psi\overline{w},f_2\rangle=(\overline{\mathcal{T}^\ast v},f_2)_{\mathcal{F}^\perp}=0,\quad f_2\in \mathcal{F}^\perp,
\]
that is,
\[
\int_{\partial\Omega}\psi w f ds=0,\quad f\in \mathcal{E}\, (=\mathcal{F}\oplus \mathcal{F}^\perp).
\]
Hence, $\psi w=0$ in $\partial\Omega$ and therefore
\[
(-\Delta+q-\lambda)w=0\; \text{in}\; \Omega_1,\quad \partial_\nu w_{|\Gamma_0}=\mp ia\sqrt{\lambda} w_{|\Gamma_0}=0.
\]
Then, $w=0$ in $\Omega_1$ according to the unique continuation property. Taking $u\in\mathscr{S}_{q,\lambda}^0$ and applying Green's formula, we obtain
\[
(u,v)_{L^2(\Omega_0)}=(u,(-\Delta+q-\lambda)w)_{L^2(\Omega_0)}=((-\Delta+q-\lambda)u,w)_{L^2(\Omega_0)}=0.
\]
Whence, $v\in(\mathscr{S}_{q,\lambda}^0)^\perp=H^\perp=\{0\}$ implying that $v=0$ and $\mbox{Ker}(\mathcal{T}^\ast)=\{0\}$.

We claim that $\mathcal{T}$ is compact. Indeed, if $(\varphi_j)\subset \mathcal{F}^\perp$ is a bounded sequence, then $\left(u_{q,\lambda}^\pm(0,\psi\overline{\varphi_j})\right)\subset H^2(\Omega)$ is a bounded sequence according to \eqref{(39)_CAC_2023}. Subtracting if necessary a subsequence, we assume that $\left(u_{q,\lambda}^\pm(0,\psi\overline{\varphi_j})\right)$ converges weakly in $H^2(\Omega)$ and strongly in $L^2(\Omega)$ to $u\in L^2(\Omega)$. Since for all $j$, $u_{q,\lambda}^\pm(0,\psi\overline{\varphi_j})_{|\Omega_0}\in\mathscr{S}_{q,\lambda}^0$, we conclude that $u\in H$.

In other words, we proved that $\mathcal{T}^\ast \mathcal{T}\colon \mathcal{F}^\perp\rightarrow \mathcal{F}^\perp$ is compact self-adjoint and positive definite operator and herefore it is diagonalizable. Hence, there exists a sequence of positive numbers $(\mu_j)$ and an orthonormal basis $(\psi_j)\subset \mathcal{F}^\perp$ so that
\[
\mathcal{T}^\ast \mathcal{T}\psi_j=\mu_j\psi_j.
\]
Define $\tau_j:=\mu_j^{1/2}$ and $u_j:=\mu_j^{-1/2}\mathcal{T}\psi_j\in H$. We find that $(u_j)$ is an orthonormal basis of $H$ as in the proof of Theorem \ref{Runge_approx} by using the boundedness of $\mathcal{T}$ and the density of the range of $\mathcal{T}$.

Since $\mathcal{T}^\ast u_j=\tau_j\psi_j$, we find
\begin{equation}\label{varphi_j_2}
\|\mathcal{T}^\ast u_j\|_{\mathcal{F}^\perp}=\tau_j.
\end{equation}

Let $t>0$, $a_j:=(u,u_j)_{L^2(\Omega_0)}$, $u=\sum_j a_j u_j\in H$ and set
\[
\varphi_t:=\sum_{j\in N_t^1}\tau_j^{-1}a_j\psi_j,
\]
where $N_t^1:=\{j;\; \tau_j>t\}$. We have
\begin{equation}\label{varphi_t_2}
\|\varphi_t\|_{\mathcal{F}^\perp}^2=\sum_{j\in N_t^1}\tau_j^{-2}|a_j|^2\le t^{-2}\sum_j|a_j|^2=t^{-2}\|u\|_{L^2(\Omega_0)}^2.
\end{equation}

Set $N_t^0:=\{j;\; \tau_j\le t\}$ and
\[
v_t:=\sum_{j\in N_t^0}a_j u_j.
\]
Since $u-\mathcal{T}\varphi_t=v_t$ and $\mathcal{T}\varphi_t\perp v_t$, we find
\[
\|v_t\|_{L^2(\Omega_0)}^2=(v_t+\mathcal{T}\varphi_t,v_t)_{L^2(\Omega_0)}=(u,v_t)_{L^2(\Omega_0)}.
\]
Let $w_t:=u_{q,\lambda}^\mp(\chi v_t,0)$. Then Green's formula yields
\begin{align*}
\|v_t\|_{L^2(\Omega_0)}^2&=(u,(-\Delta+q-\lambda)w_t)_{L^2(\Omega_0)}
\\
&=(u,\pm ia\sqrt{\lambda} w_t)_{L^2(\partial\Omega_0)}+(\partial_\nu u,w_t)_{L^2(\partial\Omega_0)}
\\
&\le \sqrt{\lambda}\|a\|_{L^\infty(\partial\Omega_0)}\|u\|_{L^2(\partial\Omega_0)}\|w_t\|_{L^2(\partial\Omega_0)}+\|\partial_\nu u\|_{L^2(\partial\Omega_0)}\|w_t\|_{L^2(\partial\Omega_0)}.
\end{align*}

Since the trace mapping
\[
H^1(\Omega_1)\ni h\mapsto h_{|\partial\Omega_0}\in L^2(\partial\Omega_0)
\]
is bounded, we find a constant $c_0=c_0(\varsigma,a)>0$ such that
\begin{equation}\label{trace_inequality_2}
c_0\|v_t\|_{L^2(\Omega_0)}^2\le \sqrt{\lambda}\|u\|_{H^2(\Omega_0)}\|w_t\|_{H^1(\Omega_1)}.
\end{equation}

Pick $0<\varepsilon<1$ arbitrarily. By Corollary \ref{Cor2.5_CAC_2023_2}, in which $\Gamma$ is replaced by $\Gamma_0$, we have
\[
C\|w_t\|_{H^1(\Omega_1)}\le\lambda^{3/2}\varepsilon^{1/8} \|v_t\|_{L^2(\Omega_0)}+\sqrt{\lambda}e^{e^{c/\varepsilon}}\|w_t\|_{H^1(\Gamma_0)},
\]
that is,
\[
C\|w_t\|_{H^1(\Omega_1)}\le\lambda^{3/2}\varepsilon^{1/8} \|v_t\|_{L^2(\Omega_0)}+\sqrt{\lambda}e^{e^{c/\varepsilon}}\|\psi w_t\|_{H^1(\partial\Omega)}.
\]
Here and henceforth, $C=C(\varsigma,a,\Gamma)>0$ and $c=c(\varsigma,\Gamma)>0$ are the generic constants.

On the other hand, applying the interpolation inequality in \cite[Chapter 1, Theorem 7.7]{Lions1972I}, we obtain
\[
\tilde{c}\|\psi w_t\|_{H^1(\partial\Omega)}\le \|\psi w_t\|_{H^{-1/2}(\partial\Omega)}^{1/4}\|\psi w_t\|_{H^{3/2}(\partial\Omega)}^{3/4},
\]
where $\tilde{c}=\tilde{c}(n,\Gamma)>0$ is a constant. Also, we have
\begin{align*}
\|\psi w_t\|_{H^{-1/2}(\partial\Omega)}&=\sup_{f\in \mathcal{E}\setminus\{0\}}\frac{|(\psi w_t,f)_{L^2(\partial\Omega)}|}{\|f\|_\mathcal{E}}
\\
&=\sup_{\substack{f_1\in \mathcal{F}\setminus\{0\},\\ f_2\in \mathcal{F}^\perp\setminus\{0\}}}\frac{|(\psi w_t,f_2)_{L^2(\partial\Omega)}|}{\sqrt{\|f_1\|_\mathcal{F}^2+\|f_2\|_{\mathcal{F}^\perp}^2}}
\\
&\le \sup_{f_2\in \mathcal{F}^\perp\setminus\{0\}}\frac{|(\psi w_t,f_2)_{L^2(\partial\Omega)}|}{\|f_2\|_{\mathcal{F}^\perp}}=\|\psi w_t\|_{(\mathcal{F}^\perp)^\ast}=\|\mathcal{T}^\ast v_t\|_{\mathcal{F}^\perp}.
\end{align*}
This, Lemma \ref{multiplication_operator} with $j=2$, \eqref{(39)_CAC_2023} and \eqref{varphi_j_2} imply
\begin{align*}
c_1\|\psi w_t\|_{H^1(\partial\Omega)}&\le \|\mathcal{T}^\ast v_t\|_{\mathcal{F}^\perp}^{1/4}\|v_t\|_{L^2(\Omega_0)}^{3/4}
\\
&\le \lambda^{3/8} t^{1/4}\|v_t\|_{L^2(\Omega_0)}\le \lambda_0^{-5/8}\lambda t^{1/4}\|v_t\|_{L^2(\Omega_0)},
\end{align*}
where $c_1=c_1(\zeta_0,a,\Gamma)>0$ is a generic constant.

In consequence, we have
\[
C\|w_t\|_{H^1(\Omega_1)}\le\lambda^{3/2}\left(\varepsilon^{1/8}+t^{1/4}e^{e^{c/\varepsilon}}\right) \|v_t\|_{L^2(\Omega_0)},
\]
which, combined with \eqref{trace_inequality_2}, implies
\[
C\|v_t\|_{L^2(\Omega_0)}\le\lambda^2\left(\varepsilon^{1/8}+t^{1/4}e^{e^{c/\varepsilon}}\right) \|u\|_{H^2(\Omega_0)}.
\]

Taking $t:=\varepsilon^{1/2} e^{-4e^{c/\varepsilon}}$ in this inequality, we deduce
\begin{equation}\label{u-v_2}
C\|v_t\|_{L^2(\Omega_0)}\le\lambda^2\varepsilon^{1/8}\|u\|_{H^2(\Omega_0)}
\end{equation}
and in light of \eqref{varphi_t_2} we obtain
\begin{equation}\label{v_2}
\|\varphi_t\|_{\mathcal{F}^\perp}\le e^{e^{c/\varepsilon}}\|u\|_{L^2(\Omega_0)}.
\end{equation}

Let $v:=u_{q,\lambda}^\pm(0,\psi\overline{\varphi_t})\in \mathscr{S}_{q,\lambda}^2$. Then, it follows from \eqref{(39)_CAC_2023} that 
\begin{equation}\label{xx}
\|v\|_{H^2(\Omega)}\le c_2\sqrt{\lambda}\|\varphi_t\|_{\mathcal{E}},
\end{equation}
 with a constant $c_2=c_2(\zeta_0,a)>0$. Since $u-v_{|\Omega_0}=u-\mathcal{T}\varphi_t=v_t$, we combine \eqref{u-v_2}, \eqref{v_2} and \eqref{xx} to complete the proof.
\end{proof}

\subsection{Integral inequality}

Fix $\kappa>0$ arbitrarily and recall that $\zeta_0=(n,\Omega,\kappa)$, $\varsigma=(n,\Omega,\Omega_0,\kappa)$ and $\mathbf{b}_\lambda=\sqrt{2\cosh(\sqrt{\lambda}/2)}$, $\lambda>0$. Let the constant $\lambda_0=\lambda_0(\zeta_0,a)>0$ be as in Lemma \ref{auxiliary_estimates}.

\begin{prop}\label{integral_inequality_2}
There exist constants $C=C(\varsigma,a,\Gamma,\Sigma)>0$, $c=c(\varsigma,\Gamma)>0$ and $\theta=\theta(\varsigma,\Sigma)\in(0,1)$ such that for all $\lambda\ge\lambda_0$, $q_j\in\mathscr{Q}$, $j=1,2$ satisfying $q_1=q_2$ in $\Omega_1$, $0<\varepsilon<1$ and $u_j\in \mathscr{S}_{q_j,\lambda}^0$, $j=1,2$, we have
\begin{align*}
&C\left|\int_{\Omega_0}(q_1-q_2)u_1 u_2 d\mu\right|\le \lambda^4\varepsilon^{1/8} \|u_1\|_{H^2(\Omega_0)}\|u_2\|_{H^2(\Omega_0)}
\\
&\hskip 4cm +\lambda^2\mathbf{b}_\lambda e^{e^{c/\varepsilon}}\|\mathscr{N}_{q_1,\lambda}-\mathscr{N}_{q_2,\lambda}\|^\theta\|u_1\|_{L^2(\Omega_0)}\|u_2\|_{L^2(\Omega_0)}.
\end{align*}
\end{prop}

\begin{proof}
Let $q_1,q_2\in\mathscr{Q}$ satisfying $q_1=q_2$ in $\Omega_1$. Pick $0<\varepsilon<1$ and $u_j\in\mathscr{S}_{q_j,\lambda}^0$, $j=1,2$. By Theorem \ref{Runge_approx_2}, there exists $v_j\in\mathscr{S}_{q_j,\lambda}^\pm$, $j=1,2$, such that
\begin{align}
\label{u_j-v_j_2}C\|u_j-{v_j}_{|\Omega_0}\|_{L^2(\Omega_0)}&\le \lambda^2\varepsilon^{1/8}\|u_j\|_{H^2(\Omega_0)},
\\
\label{v_j_2}C'\|v_j\|_{H^2(\Omega)}&\le \sqrt{\lambda}e^{e^{c/\varepsilon}}\|u_j\|_{L^2(\Omega_0)}.
\end{align}
Here, $C=C(\varsigma,a,\Gamma)>0$, $C'=C'(\zeta_0,a)>0$ and $c=c(\varsigma,\Gamma)>0$ are the constants in Theorem \ref{Runge_approx_2}. By using \eqref{u_j-v_j_2}, the identity
\[
u_1u_2=(u_1-v_1)u_2+(v_1-u_1)(u_2-v_2)+u_1(u_2-v_2)+v_1v_2
\]
and $0<\varepsilon<1$, we have
\begin{align}\label{(47)_CAC_2023}
&C\left|\int_{\Omega_0}(q_1-q_2)u_1 u_2d\mu\right|\le \lambda^4\varepsilon^{1/8}\|u_1\|_{H^2(\Omega_0)}\|u_2\|_{H^2(\Omega_0)}
\\
&\hskip 7cm +\left|\int_{\Omega_0}(q_1-q_2)v_1 v_2 d\mu\right|.\notag
\end{align}

Let $\widetilde{v}_2\in H^2(\Omega)$ be the unique solution to the BVP
\[
\begin{cases}
(-\Delta+q_1-\lambda)\widetilde{v}_2=0\quad &\text{in}\; \Omega,
\\
(\partial_\nu\mp ia\sqrt{\lambda})\widetilde{v}_2=(\partial_\nu\mp ia\sqrt{\lambda})v_2\quad &\text{on}\; \partial\Omega
\end{cases}
\]
and set $v:=v_2-\widetilde{v}_2$, which satisfies
\[
\begin{cases}
(-\Delta+q_1-\lambda)v=(q_1-q_2)v_2\quad &\text{in}\; \Omega,
\\
(\partial_\nu\mp ia\sqrt{\lambda})v=0\quad &\text{on}\; \partial\Omega.
\end{cases}
\]
Pick $\psi\in C_0^\infty(\Omega;[0,1])$ so that $\psi=1$ in an open neighborhood of $\overline{\Omega_0}$. Then, $w:=\psi v$ satisfies
\[
(-\Delta+q_1-\lambda)w=\psi(q_1-q_2)v_2-[\Delta,\psi]v\quad \text{in}\; \Omega.
\]
Applying Green's formula, we obtain
\[
\int_\Omega\psi(q_1-q_2)v_2v_1d\mu=\int_\Omega[\Delta,\psi]v v_1d\mu.
\]
Since there exists an open set $U\Subset\Omega_1$ such that $\mbox{supp}([\Delta,\psi]v)\subset U$ and $q_1=q_2$ in $\Omega_1$, we get
\[
\left|\int_{\Omega_0}(q_1-q_2)v_2v_1d\mu\right|=\left|\int_\Omega[\Delta,\psi]vv_1d\mu\right|\le c_0\|v\|_{H^1(U)}\|v_1\|_{L^2(\Omega)},
\]
where $c_0=c_0(n,\Omega,\Omega_0)>0$ is a constant. This inequality, combined with Theorem \ref{Thm7.1_CAC_2023} in which $\Omega$ and $q$ are replaced by $\Omega_1$ and $q_1$, \eqref{H^1_estimate} and \eqref{v_j_2}, yields
\begin{align*}
\left|\int_{\Omega_0}(q_1-q_2)v_2v_1d\mu\right|&\le C\sqrt{\lambda}\mathbf{b}_\lambda\left(\rho^p\|v\|_{H^1(\Omega_1)}+\rho^{-1}\|v\|_{H^1(\Sigma)}\right)\|v_1\|_{L^2(\Omega)}
\\
&\le C\sqrt{\lambda}\mathbf{b}_\lambda\left(\rho^p\|v_2\|_{L^2(\Omega)}+\rho^{-1}\|v\|_{H^1(\Sigma)}\right)\|v_1\|_{L^2(\Omega)}
\\
&\le C\sqrt{\lambda}\mathbf{b}_\lambda\left(\sqrt{\lambda}e^{e^{c/\varepsilon}}\rho^p\|u_2\|_{L^2(\Omega_0)}+\rho^{-1}\|v\|_{H^1(\Sigma)}\right)\|v_1\|_{L^2(\Omega)},
\end{align*}
where $\rho>0$ is taken arbitrarily, $C=C(\varsigma,a,\Gamma,\Sigma)>0$ and $p=p(\varsigma,\Sigma)>0$ are generic constants.

It follows from \eqref{v_j_2} that
\begin{align*}
\|v\|_{H^1(\Sigma)}&\le\|(\mathscr{N}_{q_1,\lambda}-\mathscr{N}_{q_2,\lambda})((\partial_\nu\mp ia\sqrt{\lambda})v_2)\|_{H^{3/2}(\Sigma)}
\\
&\le\|(\mathscr{N}_{q_1,\lambda}-\mathscr{N}_{q_2,\lambda}\|\|(\partial_\nu\mp ia\sqrt{\lambda})v_2\|_{H^{1/2}_\Gamma(\partial\Omega)}
\\
&\le C'\lambda e^{e^{c/\varepsilon}}\|\mathscr{N}_{q_1,\lambda}-\mathscr{N}_{q_2,\lambda}\|\|u_2\|_{L^2(\Omega_0)}.
\end{align*}
Hence, \eqref{v_j_2} yields
\begin{align*}
&C\left|\int_{\Omega_0}(q_1-q_2)v_2v_1d\mu\right|
\\
&\hskip 2cm \le \lambda^2\mathbf{b}_\lambda e^{e^{c/\varepsilon}}\left(\rho^p+\rho^{-1}\|\mathscr{N}_{q_1,\lambda}-\mathscr{N}_{q_2,\lambda}\|\right)\|u_1\|_{L^2(\Omega_0)}\|u_2\|_{L^2(\Omega_0)}.
\end{align*}

Taking in this inequality $\rho:=\|\mathscr{N}_{q_1,\lambda}-\mathscr{N}_{q_2,\lambda}\|^{1/(p+1)}$, we get
\[
C\left|\int_{\Omega_0}(q_1-q_2)v_2v_1d\mu\right|
\le \lambda^2\mathbf{b}_\lambda e^{e^{c/\varepsilon}}\|\mathscr{N}_{q_1,\lambda}-\mathscr{N}_{q_2,\lambda}\|^\theta\|u_1\|_{L^2(\Omega_0)}\|u_2\|_{L^2(\Omega_0)},
\]
where $\theta=\theta(\varsigma,\Sigma):=p/(p+1)\in(0,1)$. Putting together this inequality and \eqref{(47)_CAC_2023}, we complete the proof.

\end{proof}

\subsection{Proof of Theorem \ref{triple_log_stability_2}}

\begin{proof}
Define $q:=q_1-q_2$ extended by $0$ in $\mathbb{R}^n\setminus\Omega$. Pick $\eta\in\mathbb{R}^n$ and $\tau\ge\varpi$, where $\varpi=\varpi(n,\Omega_0,\kappa)\ge 1$ is the constant as in Proposition \ref{Prop6.1_CAC_2023} with $X:=\Omega_0$. By repeating the same arguments using Theorem \ref{Prop6.1_CAC_2023} as in the proof of Theorem \ref{triple_log_stability}, we have from Proposition \ref{integral_inequality_2},
\[
C|\hat{q}(\eta)|\le\tau^{-1}+\lambda^6 \varepsilon^{1/8} e^{\varkappa(|\eta|+2\tau)}+\lambda^2\mathbf{b}_\lambda e^{e^{c/\varepsilon}} e^{\varkappa(|\eta|+2\tau)}\|\mathscr{N}_{q_1,\lambda}-\mathscr{N}_{q_2,\lambda}\|^\theta.
\]
Here $C=C(\varsigma,a,\Gamma,\Sigma)$, $c=c(\varsigma,\Gamma)$ and $\theta=\theta(\varsigma,\Sigma)$, where $\varsigma=(n,\Omega,\Omega_0,\kappa)$, are the constants of Proposition \ref{integral_inequality_2}, and $\varkappa=\varkappa(n,\Omega_0,\kappa)>0$ is the constant of Proposition \ref{Prop6.1_CAC_2023}. The rest of the proof is exactly the same as that of Theorem \ref{triple_log_stability}.
\end{proof}

\appendix

\section{A priori estimates}

The following lemma gives a priori estimates for the BVP \eqref{BVP}, which is already obtained by Choulli \cite[Lemma 2.1]{Choulli2023a}. Recall that $\kappa:=\|q_0\|_{L^\infty(\Omega)}+\kappa_0>0$, $\zeta=(n,\Omega,g,\kappa)$ and $u_{q,\lambda}(\varphi)$ denotes the unique solution to the BVP \eqref{BVP}.

\begin{lem}{\cite[Lemma 2.1]{Choulli2023a}}\label{Lem2.1_CAC_2023}

\noindent
{\rm (i)} There exists a constant $c_1=c_1(\zeta)>0$ such that for all $f\in L^2(\Omega)$, $\lambda\in\rho(A_{q_0})$ satisfying $\lambda\ge 1$ and $q\in\mathscr{Q}_\lambda$ we have
\[
\|R_q(\lambda)f\|_{H^j(\Omega)}\le c_1\lambda^{j/2}\mathbf{e}_\lambda\|f\|_{L^2(\Omega)},\quad j=1,2.
\]
{\rm (ii)} There exists a constant $c_2=c_2(\zeta)>0$ such that for all $\varphi\in H^{3/2}(\partial\Omega)$, $\lambda\in\rho(A_{q_0})$ satisfying $\lambda\ge 1$ and $q\in\mathscr{Q}_\lambda$ we have
\[
\|u_{q,\lambda}(\varphi)\|_{H^2(\Omega)}\le c_2\lambda^2\mathbf{e}_\lambda\|\varphi\|_{H^{3/2}(\partial\Omega)}.
\]
\end{lem}
\begin{proof}
First, note that for any $\lambda\in \rho(A_{q_0})$, we have $\lambda\in\bigcap_{q\in\mathscr{Q}_\lambda}\rho(A_q)$. Indeed, if this were not the case, then there would exist $q\in\mathscr{Q}_\lambda$ and $0\neq u\in H^2(\Omega)\cap H^1_0(\Omega)$ such that $u=-R_{q_0}(\lambda)(q' u)$ and $\|u\|_{L^2(\Omega)}\le\|R_{q_0}(\lambda)\|\|q'\|_{L^\infty(\Omega)}\|u\|_{L^2(\Omega)}<\|u\|_{L^2(\Omega)}$, which is a contradiction.

(i) Let $f\in L^2(\Omega)$ and $u=R_q(\lambda)f$. The usual a priori estimate in $H^2(\Omega)$ (e.g., \cite[(2.3.3.1)]{Grisvard1985} with \cite[(2.3.3.8)]{Grisvard1985}) yields
\begin{align}\label{global_elliptic_estimate}
\|u\|_{H^2(\Omega)}&\le c_0\left(\|\Delta_g u\|_{L^2(\Omega)}+\|u\|_{L^2(\Omega)}\right)
\\
&\le c_0\left((\lambda+\kappa+1)\|u\|_{L^2(\Omega)}+\|f\|_{L^2(\Omega)}\right)\notag
\\
&\le c_0\lambda\left((2+\kappa)\|u\|_{L^2(\Omega)}+\|f\|_{L^2(\Omega)}\right),\notag
\end{align}
where $c_0=c_0(n,\Omega,g)>0$ is a constant.

On the other hand, we have
\[
u=\sum_{k=1}^\infty \frac{(f,\phi_q^k)_{L^2(\Omega)}}{\lambda_q^k-\lambda}\phi_q^k,
\]
where $(\lambda_q^k)$ is the sequence of eigenvalues of $A_q$ satisfying
\[
-\infty<\lambda_q^1\le\cdots\le \lambda_q^k\le\cdots\; \text{and}\; \lim_{k\to\infty}\lambda_q^k=\infty
\]
and $(\phi_q^k)\subset H^2(\Omega)\cap H^1_0(\Omega)$ is an orthonormal basis of $L^2(\Omega)$ consisting of eigenfunctions. Hence, we obtain
\[
\|u\|_{L^2(\Omega)}^2=\sum_{k=1}^\infty\frac{|(f,\phi_q^k)_{L^2(\Omega)}|^2}{|\lambda_q^k-\lambda|^2}\le \mathbf{e}_\lambda^2\|f\|_{L^2(\Omega)}^2.
\]
This inequality, \eqref{global_elliptic_estimate} and $\mathbf{e}_\lambda\ge 1$ give the expected inequality for $j=2$. The case $j=1$ is deduced from this inequality, $\mathbf{e}_\lambda\ge 1$ and
\begin{align*}
\int_\Omega \sum_{j,k=1}^n g^{jk}\partial_j u\partial_k u d\mu&=\int_\Omega(\lambda-q)|u|^2d\mu+\int_\Omega f\overline{u}d\mu
\\
&\le \lambda(1+\kappa)\|u\|_{L^2(\Omega)}^2+\|u\|_{L^2(\Omega)}\|f\|_{L^2(\Omega)}.
\end{align*}

(ii)  Let $\Psi\in H^2(\Omega)$ so that $\Psi_{|\partial\Omega}=\varphi$ and $\|\Psi\|_{H^2(\Omega)}\le c_0\|\varphi\|_{H^{3/2}(\partial\Omega)}$, where $c_0=c_0(n,\Omega,g)>0$ is a constant. Since we have
\[
u_{q,\lambda}(\varphi)-\Psi=R_q(\lambda)((\Delta_g-q+\lambda)\Psi),
\]
applying the inequality in (i) for $j=2$ yields
\begin{align*}
\|u_{q,\lambda}(\varphi)\|_{H^2(\Omega)}&\le \|\Psi\|_{H^2(\Omega)}+c_1\lambda\mathbf{e}_\lambda\left(\|\Psi\|_{H^2(\Omega)}+(\lambda+\kappa)\|\Psi\|_{L^2(\Omega)}\right)
\\
&\le \lambda^2\mathbf{e}_\lambda \left(1+c_1(2+\kappa)\right) c_0\|\varphi\|_{H^{3/2}(\partial\Omega)}.
\end{align*}
Taking $c_2:=c_0\left(1+c_1(2+\kappa))\right)>0$ completes the proof.
\end{proof}

The following lemma is a result on the regularity and a priori estimate for the BVP \eqref{BVP_2}, which is already obtained by Choulli \cite[Section 7]{Choulli2023a} (see also Baskin-Spence-Wunsch \cite[Corollary 1.11]{Baskin2016}). Recall that $\zeta_0=(n,\Omega,\kappa)$ and $u_{q,\lambda}^\pm(f,\varphi)$ denotes the unique solution to the BVP \eqref{BVP_2}.

\begin{lem}\label{auxiliary_estimates}
Fix $\kappa>0$ arbitrarily. There exist constants $\lambda_0=\lambda_0(\zeta_0,a)>0$ and $c_1=c_1(\zeta_0,a)>0$ such that for all $\lambda\ge\lambda_0$, $q\in \mathscr{Q}$, $f\in L^2(\Omega)$, $\varphi\in H^{1/2}(\partial\Omega)$ we have $u=u_{q,\lambda}^\pm(f,\varphi)\in H^2(\Omega)$ and
\begin{align}
&\label{H^1_estimate}
\|u\|_{H^1(\Omega)}\le c_1\left(\|f\|_{L^2(\Omega)}+\|\varphi\|_{L^2(\partial\Omega)}\right),
\\
&\label{(39)_CAC_2023}
\|u\|_{H^2(\Omega)}\le c_1\sqrt{\lambda}\left(\|f\|_{L^2(\Omega)}+\|\varphi\|_{H^{1/2}(\partial\Omega)}\right).
\end{align}
\end{lem}

\begin{proof}
Let $u:=u_{q,\lambda}^\pm(f,\varphi)\in H^2(\Omega)$ be the unique solution to the BVP \eqref{BVP_2}. According to \cite[Theorem 1.8]{Baskin2016}, we have
\[
\|\nabla u\|_{L^2(\Omega)}+\sqrt{\lambda}\|u\|_{L^2(\Omega)}\le c_0\left(\|q u+f\|_{L^2(\Omega)}+\|\varphi\|_{L^2(\partial\Omega)}\right),
\]
where $c_0=c_0(n,\Omega,a)>0$ is a constant. Set $\lambda_0:=4c_0^2\kappa^2$. For any $\lambda\ge\lambda_0$ and $q\in\mathscr{Q}$, we have
\begin{align*}
\|\nabla u\|_{L^2(\Omega)}+\sqrt{\lambda}\|u\|_{L^2(\Omega)}&\le c_0\left(\kappa\|u\|_{L^2(\Omega)}+\|f\|_{L^2(\Omega)}+\|\varphi\|_{L^2(\partial\Omega)}\right)
\\
&\le (\sqrt{\lambda}/2)\|u\|_{L^2(\Omega)}+c_0\left(\|f\|_{L^2(\Omega)}+\|\varphi\|_{L^2(\partial\Omega)}\right),
\end{align*}
which implies
\[
\|\nabla u\|_{L^2(\Omega)}+\sqrt{\lambda}\|u\|_{L^2(\Omega)}\le 2c_0\left(\|f\|_{L^2(\Omega)}+\|\varphi\|_{L^2(\partial\Omega)}\right).
\]
From this inequality, we have
\[
\|u\|_{L^2(\Omega)}\le \frac{2c_0}{\sqrt{\lambda}}\left(\|f\|_{L^2(\Omega)}+\|\varphi\|_{L^2(\partial\Omega)}\right)
\]
and
\[
\|u\|_{H^1(\Omega)}\le 2c_0\sqrt{1+\lambda_0^{-1}}\left(\|f\|_{L^2(\Omega)}+\|\varphi\|_{L^2(\partial\Omega)}\right),
\]
from which we obtain \eqref{H^1_estimate} with $c_1=c_1(\zeta_0,a):=2c_0\sqrt{1+\lambda_0^{-1}}$.

Henceforth, $c_1=c_1(\zeta_0,a)>0$ will denote a generic constant. Putting together these inequalities and the usual a priori estimate in $H^2(\Omega)$ (e.g., \cite[(2.3.3.1)]{Grisvard1985} with \cite[(2.3.3.8)]{Grisvard1985}) yields
\begin{align*}
\|u\|_{H^2(\Omega)}&\le c_1\left(\|\Delta u\|_{L^2(\Omega)}+\|u\|_{L^2(\Omega)}+\|\partial_\nu u\|_{H^{1/2}(\partial\Omega)}\right)
\\
&\le c_1\left((\lambda+\kappa+1)\|u\|_{L^2(\Omega)}+\|f\|_{L^2(\Omega)}+\|\pm ia\sqrt{\lambda} u+\varphi\|_{H^{1/2}(\partial\Omega)}\right)
\\
&\le c_1\left(\lambda(1+(\kappa+1)/\lambda_0)\|u\|_{L^2(\Omega)}+\sqrt{\lambda}\|u\|_{H^1(\Omega)}\right.
\\
&\hskip 6cm \left.+\|f\|_{L^2(\Omega)}+\|\varphi\|_{H^{1/2}(\partial\Omega)}\right)
\\
&\le c_1\sqrt{\lambda}\left(\|f\|_{L^2(\Omega)}+\|\varphi\|_{H^{1/2}(\partial\Omega)}\right).
\end{align*}
The proof is then complete.
\end{proof}

\section{Local quantitative uniqueness of continuation}

Fix $\kappa>0$ arbitrarily and recall that $\zeta=(n,\Omega,g,\kappa)$ and
\[
H^1_{0,\Sigma}(\Omega)=\{u\in H^1(\Omega);\; u_{|\Sigma}=0\}.
\]

\begin{thm}{\cite[Theorem 2.2]{Choulli2023a}}\label{Thm2.2_CAC_2023}
Let $U\Subset\Omega$ be a nonempty open set. There exist constants $p=p(\zeta,\Sigma,U)>0$ and $C=C(\zeta,\Sigma,U)>0$ such that for all $\lambda\ge 1$, $q\in L^\infty(\Omega)$ satisfying $\|q\|_{L^\infty(\Omega)}\le\kappa$, $\varepsilon>0$ and $u\in H^2(\Omega)\cap H^1_{0,\Sigma}(\Omega)$ we have
\[
C\|u\|_{H^1(U)}\le\mathbf{b}_\lambda\left(\sqrt{\lambda}\varepsilon^p\|u\|_{H^1(\Omega)}+\varepsilon^{-1}\left(\|(\Delta_g-q+\lambda)u\|_{L^2(\Omega)}+\|\partial_{\nu_g} u\|_{L^2(\Sigma)}\right)\right).
\]
\end{thm}

\begin{proof}
According to \cite[Proposition A.3]{Bellassoued2023}, there exist constants $C=C(\zeta,\Sigma,U)>0$ and $p=p(\zeta,\Sigma,U)>0$ such that for any $\varepsilon>0$ and $v\in H^2(\Omega\times(0,1))$ satisfying $v_{|\Sigma\times(0,1)}=0$, we have
\begin{align*}
&C\|v\|_{H^1(U\times(1/4,3/4))}\le \varepsilon^p\|v\|_{H^1(\Omega\times(0,1))}
\\
&\hskip 3cm +\varepsilon^{-1}\left(\|(\Delta_g+\partial_t^2-q)v\|_{L^2(\Omega\times(0,1))}+\|\partial_{\nu_g} v\|_{L^2(\Sigma\times(0,1))}\right).
\end{align*}

Henceforth, $C=C(\zeta,\Sigma,U)>0$ will denote a generic constant. Since $v:=e^{\sqrt{\lambda}t}u$ with $u\in H^2(\Omega)\cap H^1_{0,\Sigma}(\Omega)$ and $\lambda\ge 1$ satisfies
\[
C\|v\|_{H^1(U\times(1/4,3/4))}\ge \left(\lambda^{-1/2}e^{\sqrt{\lambda}}\sinh(\sqrt{\lambda}/2)\right)^{1/2}\|u\|_{H^1(U)},
\]
\[
C\|v\|_{H^1(\Omega\times(0,1))}\le \sqrt{\lambda}\left(\lambda^{-1/2} e^{\sqrt{\lambda}}\sinh\sqrt{\lambda}\right)^{1/2}\|u\|_{H^1(\Omega)}
\]
and
\begin{align*}
&C\left(\|(\Delta_g+\partial_t^2-q)v\|_{L^2(\Omega\times(0,1))}+\|\partial_{\nu_g} v\|_{L^2(\Sigma\times(0,1))}\right)
\\
&\hskip 1cm \le \left(\lambda^{-1/2}e^{\sqrt{\lambda}}\sinh\sqrt{\lambda}\right)^{1/2}\left(\|(\Delta_g-q+\lambda)u\|_{L^2(\Omega)}+\|\partial_{\nu_g} u\|_{L^2(\Sigma)}\right),
\end{align*}
by using $\frac{\sinh\sqrt{\lambda}}{\sinh(\sqrt{\lambda}/2)}=2\cosh(\sqrt{\lambda}/2)=\mathbf{b}_\lambda^2$, we obtain the expected inequality.
\end{proof}

Next, recall that
\[
\mathscr{H}^\pm=\left\{u\in H^2(\Omega);\; (\partial_\nu\mp ia\sqrt{\lambda})u_{|\Sigma}=0\right\}
\]
and $\zeta_0=(n,\Omega,\kappa)$.

\begin{thm}{\cite[Theorem 7.1]{Choulli2023a}}\label{Thm7.1_CAC_2023}
Let $U\Subset\Omega$ be a nonempty open set and $\lambda_0=\lambda_0(\zeta_0,a)>0$ be the constant of Lemma \ref{auxiliary_estimates}. There exist constants $p=p(\zeta_0,\Sigma,U)>0$ and $C=C(\zeta_0,a,\Sigma,U)>0$  such that for all  $\lambda\ge\lambda_0$, $q\in L^\infty(\Omega)$ satisfying $\|q\|_{L^\infty(\Omega)}\le\kappa$, $\varepsilon>0$ and $u\in \mathscr{H}^\pm$ we have
\[
C\|u\|_{H^1(U)}\le\sqrt{\lambda}\mathbf{b}_\lambda\left(\varepsilon^p\|u\|_{H^1(\Omega)}+\varepsilon^{-1}\left(\|(\Delta-q+\lambda)u\|_{L^2(\Omega)}+\|u\|_{H^1(\Sigma)}\right)\right).
\]
\end{thm}
\begin{proof}
According to \cite[Proposition A.3]{Bellassoued2023}, there exist constants $C=C(\zeta_0,\Sigma,U)>0$ and $p=p(\zeta_0,\Sigma,U)>0$ such that for any $\varepsilon>0$ and $v\in H^2(\Omega\times(0,1))$, we have
\begin{align*}
&C\|v\|_{H^1(U\times(1/4,3/4))}
\\
&\hskip 1cm \le \varepsilon^p\|v\|_{H^1(\Omega\times(0,1))}+\varepsilon^{-1}\left(\|(\Delta+\partial_t^2-q)v\|_{L^2(\Omega\times(0,1))}\right.
\\
&\hskip 5cm \left.+\|\partial_\nu v\|_{L^2(\Sigma\times(0,1))}+\|v\|_{H^1(\Sigma\times(0,1))}\right).
\end{align*}

Since $v:=e^{\sqrt{\lambda}t}u$ with $u\in \mathscr{H}^\pm$ satisfies $\partial_\nu v=\pm ia\sqrt{\lambda} e^{\sqrt{\lambda}t} u$ on $\Sigma\times(0,1)$ and
\begin{align*}
&C\left(\|(\Delta+\partial_t^2-q)v\|_{L^2(\Omega\times(0,1))}+\|\partial_\nu v\|_{L^2(\Sigma\times(0,1))}+\|v\|_{H^1(\Sigma\times(0,1))}\right)
\\
&\hskip 1cm \le \sqrt{\lambda}\left(\lambda^{-1/2}e^{\sqrt{\lambda}}\sinh\sqrt{\lambda}\right)^{1/2}\left(\|(\Delta-q+\lambda)u\|_{L^2(\Omega)}+\|u\|_{H^1(\Sigma)}\right),
\end{align*}
where $C$ depends also on $\lambda_0$ and $a$, we complete the proof in the same way as that of Theorem \ref{Thm2.2_CAC_2023}.
\end{proof}

\section{Complex geometric optics solutions}\label{sectionC}

Let $X$ be a bounded domain of $\mathbb{R}^n$ and define for $\lambda\ge 0$ and $q\in L^\infty(X)$
\[
\mathscr{S}_{q,\lambda}^0=\left\{u\in H^2(X);\; (-\Delta+q-\lambda)u=0\quad \text{in}\; X\right\}.
\]

Fix $\kappa>0$ arbitrarily. The complex geometric optics solutions we used are contained in the following proposition (see also \cite[Proposition 6.1]{Choulli2023a} and \cite[Proposition C.2]{Krupchyk2019}).

\begin{prop}\label{Prop6.1_CAC_2023}
There exist constants $\varpi=\varpi(n,X,\kappa)\ge 1$, $C=C(n,X,\kappa)>0$ and $\varkappa=\varkappa(n,X,\kappa)>0$ such that for all $\lambda\ge 1$, $q\in L^\infty(\Omega)$ satisfying $\|q\|_{L^\infty(X)}\le \kappa$ and $\xi\in\mathbb{C}^n$ so that $\xi\cdot\xi=\lambda$ and $|\Im\xi|\ge\varpi$ we find 
\[
u_\xi=e^{-ix\cdot\xi}(1+w_\xi)\in\mathscr{S}_{q,\lambda}^0
\] 
in such a way that
\begin{align*}
&C\|w_\xi\|_{L^2(X)}\le |\Im\xi|^{-1},
\\
&C\|u_\xi\|_{L^2(X)}\le e^{\varkappa|\Im\xi|},
\\
&C\|u_\xi\|_{H^2(X)}\le \lambda e^{\varkappa|\Im\xi|}.
\end{align*}
\end{prop}

\begin{proof}
Let $\xi\in\mathbb{C}^n$ satisfying $\xi\cdot\xi=\lambda$. We first note that if $u_\xi=e^{-ix\cdot\xi}(1+w_\xi)\in \mathscr{S}_{q,\lambda}^0$, then $w_\xi$ must satisfies 
\[
-\Delta w_\xi+2i\xi\cdot \nabla w_\xi=-q(1+w_\xi)\quad \text{in}\; X.
\]

Before solving this equation, we introduce several notations. Set $D_j:=-i\partial_j$, $D:=(D_1,\ldots,D_n)$ and $P_\xi(\eta):=\eta\cdot\eta-2\xi\cdot\eta$, $\eta=(\eta^1,\ldots,\eta^n)\in\mathbb{R}^n$ so that $P_\xi(D)=-\Delta+2i\xi\cdot\nabla$. Let 
\[
\widetilde{P}_\xi(\eta):=\left(\sum_{|\alpha|\le 2}|D^\alpha P_\xi(\eta)|^2\right)^{1/2}.
\]

Also, for $1\le j,k\le n$, set $Q_j(\eta):=\eta^j$ and $Q_{j,k}(\eta):=\eta^j\eta^k$ and note that $|Q_k(\eta)|\le|\eta|$ and $|Q_{j,k}(\eta)|\le |\eta|^2$ hold. For $P_\xi(\eta)=\eta\cdot\eta-2\xi\cdot\eta$, we have
\begin{align*}
\widetilde{P}_\xi(\eta)^2\ge |P_\xi(\eta)|^2+\sum_j|D_j P_\xi(\eta)|^2&\ge \left(|\eta|^2-2|\Re\xi||\eta|\right)^2+4|\Im\xi|^2
\\
&\ge\begin{cases}|\eta|^4/4,\quad &4|\Re\xi|< |\eta|,
\\
4|\Im\xi|^2,\quad &4|\Re\xi|\ge |\eta|.
\end{cases}
\end{align*}
Assume that $|\Im\xi|\ge 1$. As $\xi\cdot\xi=\lambda$ and $\lambda\ge 1$, we have  
\[
|\Im\xi|\le |\Im\xi|^2\le |\Re\xi|^2=|\Im\xi|^2+\lambda\le (|\Im\xi|+\sqrt{\lambda})^2\le 4\lambda|\Im\xi|^2.
\]
Then, we find 
\begin{align}
&\sup_{\eta\in\mathbb{R}^n}\frac{|Q_j(\eta)|}{\widetilde{P}_\xi(\eta)}\le\max((2|\Re\xi|)^{-1},2|\Re\xi|)= 2|\Re\xi|\le 4\lambda|\Im\xi| \label{y1}
\\
&\sup_{\eta\in\mathbb{R}^n}\frac{|Q_{j,k}(\eta)|}{\widetilde{P}_\xi(\eta)}\le \max(2,8|\Re\xi|^2/|\Im\xi|)=8|\Re\xi|^2/|\Im\xi|\le 32\lambda |\Im\xi|. \label{y2}
\end{align}

By \cite[Th\'{e}or\`{e}me 2.3]{Choulli2009} there exists $E_\xi\in \mathscr{B}(L^2(X))$ and $c_0=c_0(n,X)>0$ such that
\[
P_\xi(D)E_\xi f=f,\quad f\in L^2(X)
\]
and
\begin{equation}\label{y3}
\|E_\xi\|_{\mathscr{B}(L^2(X))}\le c_0\sup_{\eta\in\mathbb{R}^n}\frac{1}{\widetilde{P}_\xi(\eta)}\le c_0\sup_{\eta\in\mathbb{R}^n}\frac{1}{\sqrt{\sum_j|D_j P_\xi(\eta)|^2}}=\frac{c_0}{2|\Im\xi|}.
\end{equation}
Furthermore, in light of \eqref{y1} and \eqref{y2}, for all $f\in L^2(X)$, $E_\xi f\in H^2(X)$ and
\begin{equation}\label{y4}
\|E_\xi f\|_{H^2(X)}\le 32c_0\lambda |\Im\xi|\|f\|_{L^2(X)}.
\end{equation}

Next, consider the operator
\[
F_\xi\colon L^2(X)\ni f\mapsto E_\xi[-q(1+f)]\in L^2(X).
\]
We have
\[
\|F_\xi f-F_\xi g\|_{L^2(X)}=\|E_\xi[-q(f-g)]\|_{L^2(X)}\le\frac{c_0\kappa}{2|\Im\xi|}\|f-g\|_{L^2(X)}.
\]
In consequence, if $|\Im\xi|\ge \varpi:=\max(c_0\kappa,1)(>c_0\kappa/2)$, then $F_\xi$ admits a unique fixed point $w_\xi\in L^2(X)$. Therefore, $w_\xi =E_\xi[-q(1+w_\xi)]\in H^2(X)$ and applying $P_\xi(D)$ to the both sides of the identity $w_\xi =E_\xi[-q(1+w_\xi)]$ yields
\[
-\Delta w_\xi+2i\xi\cdot\nabla w_\xi=P_\xi(D)E_\xi[-q(1+w_\xi)]=-q(1+w_\xi)\in L^2(X).
\]
Using \eqref{y3}, we deduce
\begin{align*}
\|w_\xi\|_{L^2(X)}=\|F_\xi w_\xi\|_{L^2(X)}&\le\frac{c_0\kappa}{2|\Im\xi|}\|1+w_\xi\|_{L^2(X)}
\\
&\le \frac{c_0\kappa}{2|\Im\xi|}\left(\sqrt{|X|}+\|w_\xi\|_{L^2(X)}\right)
\end{align*}
and as $c_0\kappa\le|\Im\xi|$, we obtain
\begin{equation}\label{L^2_estimate}
\|w_\xi\|_{L^2(X)}\le\frac{c_0\kappa\sqrt{|X|}}{2|\Im\xi|-c_0\kappa}\le\frac{c_0\kappa\sqrt{|X|}}{|\Im\xi|}.
\end{equation}
On the other hand, \eqref{y4} implies
\[
\|w_\xi\|_{H^2(X)}=\|E_\xi(-q(1+w_\xi))\|_{H^2(X)}\le 32c_0\lambda |\Im\xi|\|q(1+w_\xi)\|_{L^2(X)}.
\]
This and \eqref{L^2_estimate} imply
\begin{equation}\label{y5}
\|w_\xi\|_{H^2(X)}\le c_1\lambda|\Im\xi|,
\end{equation}
for some constant $c_1=c_1(n,X,\kappa)>0$.

Finally, by construction $u_\xi=e^{-ix\cdot\xi}(1+w_\xi)\in \mathscr{S}_{q,\lambda}^0$ and, using \eqref{L^2_estimate} and \eqref{y5}, we obtain the expected inequalities.
\end{proof}

\section{Boundedness of a multiplication operator}

\begin{lem}\label{multiplication_operator}
Let $j=1,2$, $\Omega\subset\mathbb{R}^n$ be a $C^{j-1,1}$ bounded domain and $\psi\in C^{j-1,1}(\overline{\Omega})$. Then, the multiplication operator
\[
H^{j-1/2}(\partial\Omega)\ni f\mapsto \psi_{|\partial\Omega}f\in H^{j-1/2}(\partial\Omega)
\]
is bounded.
\end{lem}

\begin{proof}
Let $j=1,2$. From \cite[Theorem 1.5.1.2]{Grisvard1985}, the mapping $C^{j-1,1}(\overline{\Omega})\ni u\mapsto u_{|\partial\Omega}$ has the bounded extension, denoted by $\gamma_j\colon H^j(\Omega)\rightarrow H^{j-1/2}(\partial\Omega)$. This operator $\gamma_j$ has right continuous inverse. In particular, we deduce that the natural norm of $H^{j-1/2}(\partial\Omega)$ is equivalent to the following quotient norm
\[
\|f\|_{H^{j-1/2}(\partial\Omega)}:=\inf_{F\in\dot{f}}\|F\|_{H^j(\Omega)},
\]
where $\dot{f}=\{F\in H^j(\Omega);\; \gamma_j(F)=f\}$.

Let $f\in H^{j-1/2}(\partial\Omega)$. Since $\dot{f}$ is a nonempty closed convex subset of $H^j(\Omega)$, according to the projection theorem, there exists a unique $F_f\in\dot{f}$ such that
\[
\inf_{F\in\dot{f}}\|F\|_{H^j(\Omega)}=\inf_{F\in\dot{f}}\|F-0\|_{H^j(\Omega)}=\|F_f\|_{H^j(\Omega)}.
\]
In other words, we have
\[
\|f\|_{H^{j-1/2}(\partial\Omega)}=\|F_f\|_{H^j(\Omega)}.
\]

By density, using that $\gamma_j(\psi u)=\gamma_j(\psi)\gamma_j(u)$ for all $u\in C^{j-1,1}(\overline{\Omega})$, we obtain
\[
\gamma_j(\psi u)=\gamma_j(\psi)\gamma_j(u),\quad u\in H^j(\Omega).
\]
If $f\in H^{j-1/2}(\partial\Omega)$, then
\[
\gamma_j(\psi)f=\gamma_j(\psi)\gamma_j(F_f)=\gamma_j(\psi F_f).
\]
Whence, we have
\begin{align*}
\|\gamma_j(\psi)f\|_{H^{j-1/2}(\partial\Omega)}\le \|\gamma_j\|\|\psi F_f\|_{H^j(\Omega)}&\le c_0\|\gamma_j\|\|\psi\|_{C^{j-1,1}(\overline{\Omega})}\|F_f\|_{H^j(\Omega)}
\\
&=c_0\|\gamma_j\|\|\psi\|_{C^{j-1,1}(\overline{\Omega})}\|f\|_{H^{j-1/2}(\partial\Omega)},
\end{align*}
where $\|\gamma_j\|$ is the norm of $\gamma_j$ in $\mathscr{B}(H^j(\Omega);H^{j-1/2}(\partial\Omega))$ and $c_0>0$ is an universal constant.
\end{proof}

\section*{Acknowledgement}

This work was supported by JSPS KAKENHI Grant Numbers JP25K17280, JP23KK0049.

\section*{Declarations}
\subsection*{Conflict of interest}
The authors declare that they have no conflict of interest.

\subsection*{Data availability}
Data sharing not applicable to this article as no datasets were generated or analyzed during the current study.

\printbibliography

@book{Lions1972I,
   author = {J. L. Lions and E. Magenes},
   city = {Berlin, Heidelberg},
   isbn = {9783662117750},
   journal = {Vorlesungen über numerisches Rechnen},
   pages = {xvi+357 pp},
   publisher = {Springer-Verlag, New York-Heidelberg},
   title = {Non-Homogeneous Boundary Value Problems and Applications},
   volume = {I},
   year = {1972}
}

@article{Bellassoued2023,
   author = {M. Bellassoued and M. Choulli},
   doi = {10.1007/s42985-023-00242-2},
   issn = {26622971},
   journal = {Partial Differential Equations and Applications},
   pages = {44 pp},
   publisher = {Springer International Publishing},
   title = {Global logarithmic stability of a Cauchy problem for anisotropic wave equations},
   volume = {4(3), No. 23},
   year = {2023}
}

@article{Choulli2020,
   author = {M. Choulli},
   doi = {10.1017/S0004972719000789},
   issn = {17551633},
   journal = {Bulletin of the Australian Mathematical Society},
   pages = {141-145},
   publisher = {Cambridge University Press},
   title = {New global logarithmic stability results on the Cauchy problem for elliptic equations},
   volume = {101(1)},
   year = {2020}
}

@book{Choulli2009,
   author = {M. Choulli},
   city = {Berlin},
   pages = {xxii+249 pp},
   publisher = {Springer-Verlag, Berlin},
   title = {Une introduction aux problèmes inverses elliptiques et paraboliques},
   year = {2009}
}

@article{Choulli2023a,
   author = {M. Choulli},
   doi = {10.3934/cac.2023012},
   issn = {2837-0562},
   journal = {Communications on Analysis and Computation},
   pages = {214-233},
   title = {Stability of determining the potential from partial boundary data in a Schrödinger equation in the high frequency limit},
   volume = {1(3)},
   url = {https://www.aimsciences.org//article/doi/10.3934/cac.2023012},
   year = {2023}
}

@article{Ruland2019,
   author = {A. Rüland and M. Salo},
   doi = {10.1093/imrn/rnx301},
   issn = {16870247},
   journal = {International Mathematics Research Notices},
   pages = {6216-6234},
   publisher = {Oxford University Press},
   title = {Quantitative Runge approximation and inverse problems},
   volume = {20},
   year = {2019}
}

@article{Krupchyk2019,
   author = {K. Krupchyk and G. Uhlmann},
   doi = {10.1016/j.matpur.2019.02.017},
   issn = {00217824},
   journal = {Journal des Mathematiques Pures et Appliquees},
   pages = {273-291},
   publisher = {Elsevier Masson SAS},
   title = {Stability estimates for partial data inverse problems for Schrödinger operators in the high frequency limit},
   volume = {126},
   year = {2019}
}

@article{Baskin2016,
   author = {D. Baskin and E. A. Spence and J. Wunsch},
   doi = {10.1137/15M102530X},
   issn = {10957154},
   journal = {SIAM Journal on Mathematical Analysis},
   pages = {229-267},
   publisher = {Society for Industrial and Applied Mathematics Publications},
   title = {Sharp high-frequency estimates for the Helmholtz equation and applications to boundary integral equations},
   volume = {48(1)},
   year = {2016}
}

@book{Grisvard1985,
   author = {P. Grisvard},
   pages = {xiv+410 pp},
   publisher = {Pitman (Advanced Publishing Program), Boston, MA},
   title = {Elliptic Problems in Nonsmooth Domains},
   url = {https://epubs.siam.org/terms-privacy},
   year = {1985}
}

@article{GarciaFerrero2022,
   author = {M. Á. García-Ferrero and A. Rüland and W. Zatoń},
   doi = {10.3934/ipi.2021049},
   issn = {1930-8337},
   journal = {Inverse Problems and Imaging},
   pages = {251-281},
   title = {Runge approximation and stability improvement for a partial data Calderón problem for the acoustic Helmholtz equation},
   volume = {16(1)},
   url = {https://www.aimsciences.org/article/doi/10.3934/ipi.2021049},
   year = {2022}
}

\end{document}